\documentclass[12pt,a4paper,reqno]{amsart}
\usepackage{amsmath}
\usepackage{amsfonts,amssymb,fancyhdr}
\usepackage[hidelinks]{hyperref}
\usepackage[nobysame]{amsrefs}
\makeatletter
\def\MR#1{}                
\def\@bibmrnumber#1{}
\def\@bib@mrreview#1{}
\def\@bib@mathreviews#1{}

\makeatother
\usepackage{amsthm}
\makeatletter
\renewenvironment{proof}[1][\proofname]{\par
  \pushQED{\qed}%
  \normalfont \topsep6\p@\@plus6\p@\relax
  \trivlist
  \item[\hskip\labelsep\bfseries #1\@addpunct{.}]\ignorespaces
}{%
  \popQED\endtrivlist\@endpefalse
}
\makeatother
\usepackage{tikz}
\usepackage{graphicx}
\usepackage{multirow}
\usepackage{float}
\usepackage{booktabs} 
\usepackage{xcolor}

\usepackage{geometry}

\def\Box{\vcenter{\vbox{\hrule\hbox{\vrule
     \vbox to 8.8pt{\hbox to 10pt{}\vfill}\vrule}\hrule}}}

\DeclareMathOperator{\tr}{tr}

\newcommand{\ra}{\rangle}

\newcommand{\la}{\langle}
\newcommand{\F}{{\mathbb F}}

\newcommand{\cB}{{\mathcal B}}

\newcommand{\cQ}{{\mathcal Q}}

\newcommand{\cF}{{\mathcal F}}

\newcommand{\PG}{\textup{PG}}
\newcommand{\PGL}{\textup{PGL}}
\newcommand{\POmega}{\textup{P}\Omega}
\newcommand{\GL}{\textup{GL}}

\newcommand{\SU}{\textup{SU}}
\newcommand{\SL}{\textup{SL}}
\newcommand{\SO}{\textup{SO}}
\newcommand{\PGO}{\textup{PGO}}
\newcommand{\PGaO}{\textup{P}\Gamma\textup{O}}

\newcommand{\PSL}{\textup{PSL}}

\newcommand{\bfx}{\mathbf{x}}
\newcommand{\bfy}{\mathbf{y}}
\newcommand{\bfz}{\mathbf{z}}
\newcommand{\bfu}{\mathbf{u}}
\newcommand{\bfv}{\mathbf{v}}
\newcommand{\bfw}{\mathbf{w}}
\newcommand{\bfe}{\mathbf{e}}

\newcommand{\Gso}{G_{\mathrm{0}}}
\DeclareMathOperator{\Stab}{Stab}

\DeclareMathOperator{\Aut}{Aut}

\numberwithin{equation}{section}

\theoremstyle{plain}
\newtheorem{thm}{Theorem}[section]
\newtheorem{lemma}[thm]{Lemma}
\newtheorem{proposition}[thm]{Proposition}
\newtheorem{corollary}[thm]{Corollary}

\theoremstyle{definition}

\theoremstyle{remark}

\begin{document}
\newcommand{\stopthm}{\begin{flushright}
$\box \;\;\;\;\;\;\;\;\;\; $
\end{flushright}}

\newcommand{\symfont}{\fam \mathfam}
\title{The classification of non-solvable transitive BLT-sets of $\cQ(4,q)$}
\author{Weicong~Li, Qing~Xiang$^\star$}

\address{Weicong~Li, Department of Mathematics, School of Sciences, Great Bay University, Dongguan, China.}
\email{liweicong@gbu.edu.cn}
\address{Qing Xiang, Department of Mathematics and Shenzhen International Center for Mathematics, Southern University of Science and Technology, Shenzhen, China.}
\email{xiangq@sustech.edu.cn}
\keywords{BLT-sets, flock generalized quadrangles, transitive group actions,
$\cQ(4,q)$, non-solvable groups}
\subjclass[2020]{Primary 51E12; Secondary 05E18, 51E20}
\thanks{$^\star$Corresponding author.}
\begin{abstract}
For an odd prime power $q$, a BLT-set $\cB$ is a set
of $(q+1)$ points of $\cQ(4,q)$
such that every point of $\cQ(4,q)$ is collinear with at most two points of $\cB$. We complete the classification of BLT-sets of $\cQ(4,q)$ admitting a non-solvable transitive automorphism group. Nelson and Penttila completed the classification of BLT-sets $\cB$ with a non-solvable transitive automorphism group acting irreducibly on the underlying  $5$-dimensional vector space. In this paper, we prove that a BLT-set with a non-solvable transitive full automorphism group acting reducibly on the underlying $5$-dimensional vector space is either a classical BLT-set or a Kantor semifield BLT-set.
\end{abstract}
\maketitle

\section{Introduction}\label{sec:1}
  
Let $q=p^e$ be an odd prime power, and let $\mathcal K$ be a quadratic cone in $\PG(3,q)$ with vertex $v$. A
flock of $\mathcal K$ is a set $\cF$ of $q$ planes of $\PG(3,q)$, none containing $v$,
such that the conics $\pi\cap\mathcal K$, for $\pi\in\cF$, partition
$\mathcal K\setminus\{v\}$.  Bader, Lunardon, and Thas \cite{BLT1990} introduced BLT-sets
in their study of the derivation of flocks of $\PG(3,q)$ for odd $q$. Their construction associates with a flock a set of points in
the classical generalized quadrangle $\cQ(4,q)$, namely, a \emph{BLT-set}. A set $\cB$ of $(q+1)$ points of $\cQ(4,q)$ is called a {\it BLT-set} if every point of $\cQ(4,q)$ is collinear with at most two points of $\cB$. The term ``BLT-set'' was introduced by Kantor
\cite{Kantor1991GenBLT}. A BLT-set $\cB$ gives rise to $q+1$ flocks, and
the number of non-isomorphic flocks arising from $\cB$ is
equal to the number of orbits of the stabilizer $\PGaO_5(q)_{\cB}$ on
$\cB$; see \cite{BLT1990}.

BLT-sets are closely connected with flock generalized quadrangles. Thas
\cite{Thas1987} showed that a flock of a quadratic cone gives rise to an
elation generalized quadrangle of order $(q^2,q)$, and Knarr \cite{Knarr1992} constructed the same generalized quadrangle directly from
the corresponding BLT-set. Under the duality between $\cQ(4,q)$ and
$W(3,q)$, the set $\cB$ corresponds to  a set of $q+1$ totally isotropic lines of
$W(3,q)$. If $P$ is a point of $W(5,q)$ and $P^\perp/P$ is identified with
$W(3,q)$, these lines lift to $q+1$ totally isotropic planes on $P$; in Knarr's construction, they represent the lines through the distinguished
elation point. The construction therefore depends directly on $\cB$, whereas recovering a flock requires a choice of a point of $\cB$. Payne's
theorem relates equivalence of the corresponding $q$-clans to isomorphism of
the associated generalized quadrangles \cite{Payne1996}.

BLT-sets are also related to ovoids of $Q^+(5,q)$, spreads of
$\PG(3,q)$ and the resulting translation planes \cite{Walker1976},
$q$-clans and hyperbolic fibrations \cites{BDEW1999,BEP2005}, and
hemisystems \cites{BGR2010,van2010}.

 Constructing new infinite families of BLT-sets remains difficult. Nine infinite families are recorded in \cite{Nelson-thesis}.
Five of these families consist of BLT-sets admitting transitive
automorphism groups: the classical, Kantor semifield, Fisher, Fisher--Thas--Walker, and Mondello families. A classical BLT-set is the non-degenerate conic $\pi\cap\cQ(4,q)$, where $\pi$ is a non-degenerate plane of $\PG(4,q)$ and $\pi^\perp$ is an external line to $\cQ(4,q)$. By \cite{LawPenttila2003}*{Theorem~2.4}, a BLT-set meeting a classical BLT-set
in at least $(q+1)/2$ points is either classical or Fisher. Thas showed
that a Kantor semifield BLT-set spans a hyperplane whose intersection with
$\cQ(4,q)$ is a hyperbolic quadric \cite{Thas1987}; see also \cite{LawPenttila2003}*{Theorem~2.3} or \cite{Nelson-thesis}*{Theorem 2.3.5}. A characterization of the
Fisher--Thas--Walker BLT-sets is given in \cite{Thas1993CharFTW}.

Nelson and Penttila classified the BLT-sets that admit an irreducible transitive automorphism group in \cite{NelsonPenttila2013}. All such irreducible groups are non-solvable. Apart from the Fisher--Thas--Walker family, there are two sporadic examples, found by Law and Penttila over $\F_{29}$ and $\F_{59}$, respectively; see \cite{Law2004Con}. Their full automorphism groups are $S_6$ and $S_5$, respectively. The classifications in the cases where  $q\leq29$ and $31\leq q\leq67$ are given in \cites{LawPenttila2003} and \cite{Betten2013}, respectively; the associated coordinates and stabilizers are provided in \cite{SmallBLT}.

In this paper, we consider BLT-sets $\cB$ whose full automorphism group $\overline G:=\Aut(\cB)\leq\PGaO_5(q)$ is non-solvable and acts transitively on $\cB$. Two BLT-sets are equivalent if they lie in the same
$\PGaO_5(q)$-orbit. Let $\pi:\Gamma\textup{O}_5(q)\to\PGaO_5(q)$ be the
natural projection, and let $G:=\pi^{-1}(\overline G)$ be the full preimage of $\overline G$ in $\Gamma\textup{O}_5(q)$. Set
\[
 \Gso:=G\cap\SO_5(q).
\]
Since $\mathrm{PSO}_5(q)$ is normal in $\PGaO_5(q)$ and $|\PGaO_5(q):\mathrm{PSO}_5(q)|=e$, we have 
\begin{equation}\label{eqn_orthogonal_core}
\pi(\Gso)\unlhd\overline G,
 \qquad
 [\overline G:\pi(\Gso)]\mid e.
\end{equation}

We call $\overline G$ reducible if $G$ acts reducibly on the underlying $5$-dimensional vector space, and irreducible otherwise.  Our main result is stated as follows.
  
\begin{thm}\label{thm_main}
 Let $\cB$ be a BLT-set of $\cQ(4,q)$ whose full automorphism group $\Aut(\cB)$ is non-solvable and acts transitively on $\cB$. If $\Aut(\cB)$ is reducible, then $\cB$ is either a classical BLT-set or a Kantor semifield BLT-set.
\end{thm}

Combining Theorem \ref{thm_main} with the 
classification of Nelson and Penttila \cite{NelsonPenttila2013} in the irreducible case yields the following complete classification of BLT-sets admitting non-solvable transitive automorphism groups.

\begin{thm}\label{cor:full-classification}
Let $\cB$ be a BLT-set of $\cQ(4,q)$ whose full automorphism group is
non-solvable and transitive on $\cB$. Then, up to equivalence, $\cB$ is one
of the following:
\begin{enumerate}
\item a classical BLT-set;
\item a Kantor semifield BLT-set;
\item a Fisher--Thas--Walker BLT-set;
\item the Law--Penttila sporadic BLT-set over $\F_{29}$, with
$\Aut(\cB)\cong S_6$;
\item the Law--Penttila sporadic BLT-set over $\F_{59}$, with
$\Aut(\cB)\cong S_5$.
\end{enumerate}
\end{thm}

The paper is organized as follows. In Section \ref{sec:2}, we collect the geometric and group-theoretic preliminaries, including the quadratic-character sum estimates and explicit representations that will be used later. We handle the case in which $\Gso$ has a composition factor isomorphic to $\PSL_2(q_0)$ in Section \ref{sec:3}, whereas Section \ref{sec:4} deals with the exceptional $A_5$ case. For the $A_5$ case, we first use the character-sum estimate derived from the Weil bound (Theorem \ref{thm_FFquadsum}) to reduce the problem to finitely many values of $q$, and then use exhaustive exact finite-field computations to eliminate the remaining cases. Finally, in Section \ref{sec:5}, we interpret the classification in terms of flock generalized quadrangles and discuss some remaining open problems.

\section{Preliminaries}\label{sec:2} 
Throughout the paper, let $p$ be an odd prime and $q=p^e$, where $e\geq1$. Let
$\F_q$ be the finite field of order $q$, and write
$ \F_{q}^{\times}=\F_q\setminus\{0\}$. Let
\[
\square_q:=\{x^2:x\in \F_{q}^{\times}\},
\qquad
\blacksquare_q:=\F_{q}^{\times}\setminus\square_q,
\]
denote the sets of nonzero squares and nonsquares in $\F_q$,
respectively. For $a\in \F_{q}^{\times}$, write $a\square_q:=\{as:s\in\square_q\}$ and
$a\blacksquare_q:=\{an:n\in\blacksquare_q\}$. Let
$\chi: \F_{q}^{\times}\to\{\pm1\}$ be the quadratic character, extended to the whole field by setting $\chi(0)=0$. Let $V$ be a five-dimensional vector space over $\F_q$ with a non-degenerate quadratic form $Q$. The polarization of $Q$ is a bilinear form $B$ defined as follows:
\[
B(\bfx,\bfy):=Q(\bfx+\bfy)-Q(\bfx)-Q(\bfy),
\qquad \forall\, \bfx,\bfy\in V.
\]

For each nonzero vector $\bfv\in V$, write $\la\bfv\ra$ for the projective
point $\F_q\bfv$ in the projective space $\PG(V)=\PG(4,q)$. When applied to group elements, the angle brackets denote the subgroup they generate. A point $\la\bfv\ra$ is called \emph{singular} if $Q(\bfv)=0$. A (projective) line is called \emph{totally singular} if all its points are singular. These singular points and totally singular lines form the classical generalized quadrangle $\cQ(4,q)$ of order $(q,q)$. A projective line $\ell$ of $\PG(V)$ is called \emph{external} to $\cQ(4,q)$ if $\ell\cap \cQ(4,q)=\emptyset$. Equivalently, if $\ell =\PG(U)$, then the restriction $Q|_{U}$ is called \emph{anisotropic}: $Q(\bfu)\ne 0$ for
every ${\bf 0} \neq \bfu\in U$. For a vector subspace $U\leq V$, we write 
\[
U^\perp:=\{\bfx\in V:B(\bfx,\bfu)=0,\, \forall\,\bfu\in U\},
\]
and the same notation is used for the corresponding projective subspace. Let $\operatorname{disc}(Q)$ denote the determinant of the Gram matrix of $B$ with respect to a chosen basis. The square class of $\operatorname{disc}(Q)$ is independent of the choice of basis. Write $\Gamma\textup{O}_5(q)$ for the full semisimilarity group of $Q$, $\PGaO_5(q)$ for its projective image, and $
\pi:\Gamma\textup{O}_5(q)\longrightarrow\PGaO_5(q)$ for the natural projection. For $X\leq\Gamma\textup{O}_5(q)$, set
$\overline X:=\pi(X)$. Throughout the paper, vectors in $V$ are written as column vectors and matrices act on vectors from the left. Thus a matrix $g$ maps $\bfv$ to $g\bfv$, and $gP:=\la g\bfv\ra$ for $P=\la\bfv\ra$. If $X$ acts on $V$ or
$\PG(V)$, we write
\[
\bfv^X:=\{x\bfv:x\in X\} \text{ and }
P^X:=\{xP:x\in X\}.
\]

Recall that a partial BLT-set $\cB$ is a set of pairwise non-collinear
points of $\cQ(4,q)$ such that no point of $\cQ(4,q)$ is collinear with
three points of $\cB$. Bader, O'Keefe, and Penttila gave the following
algebraic criterion for partial BLT-sets; see \cite{BOP2004some}*{Corollary 4.2}.
\begin{lemma}\label{lem_disc}
Let $\cB$ be a set of at least three points of $\cQ(4,q)$, where
$q$ is an odd prime power. Then $\cB$ is a partial BLT-set if and only if, for every triple of distinct points $\la\bfx\ra,\la\bfy\ra,\la\bfz\ra$ of $\cB$,
\[
-\frac{2B(\bfx,\bfy)B(\bfx,\bfz)B(\bfy,\bfz)}
{\operatorname{disc}(Q)}\in\blacksquare_q.
\]
Moreover, if $|\cB|=q+1$ and the above condition holds for every triple of distinct points, then $\cB$ is a BLT-set.
\end{lemma}
Law and Penttila \cite{Law2004Con} used this criterion to search for BLT-sets over small fields. Complete classifications in the cases where $q\leq29$, and $31\leq q\leq67$ were given in \cite{LawPenttila2003} and \cite{Betten2013},
respectively. The relevant coordinates and full stabilizers are
available on the webpage \cite{SmallBLT}. 

\begin{thm}\label{thm_square_class}
Let $P=\la\bfv\ra$ be a point of a BLT-set $\cB$. Let
$H\leq\SO_5(q)$ with $\overline{H}$ preserving $\cB$, and assume that 
$\Stab_H(\la\bfv\ra)=\Stab_H(\bfv)$. Then
\[
\epsilon_{P,H}:H\to\{\pm1\},\qquad
\epsilon_{P,H}(h)=
\begin{cases}
-\chi\bigl(B(\bfv,h\bfv)\bigr)
 \chi\bigl(-\operatorname{disc}(Q)/2\bigr),
&\text{if }hP\ne P,\\
1,&\text{if }hP=P,
\end{cases}
\]
is a group homomorphism. In particular, if $L\leq H$ is perfect and
$g\in L$ satisfies
$gP\ne P$, then
\[
B(\bfv,g\bfv)\in-\frac{\operatorname{disc}(Q)}{2}\blacksquare_q.
\]
\end{thm}
\begin{proof}
To simplify notation, write $\epsilon=\epsilon_{P,H}$. We show that
$\epsilon(ab)=\epsilon(a)\epsilon(b)$ for any $a,b\in H$. Let $a,b\in H$. If $P$, $aP$, and
$abP$ are pairwise distinct, then Lemma \ref{lem_disc} and
$B(a\bfv,ab\bfv)=B(\bfv,b\bfv)$ give
\[
\epsilon(a)\epsilon(b)\epsilon(ab)=1.
\]
If two of these three points coincide, a direct verification shows that the above equality holds by the definition of $\epsilon$ and the fact that $\SO_5(q)$ preserves the form $B$. Therefore,
$\epsilon(ab)=\epsilon(a)\epsilon(b)$ for all $a,b\in H$, and hence
$\epsilon$ is a group homomorphism.

Since $L$ is perfect, $L=[L,\,L]$, and the restriction of $\epsilon$ to $L$
is trivial. Hence $\epsilon(g)=1$, so
$\chi(B(\bfv,g\bfv))=-\chi(-\operatorname{disc}(Q)/2)$, which is equivalent to
$B(\bfv,g\bfv)\in-\frac{\operatorname{disc}(Q)}{2}\blacksquare_q$.
\end{proof}

Based on the correspondence between BLT-sets and flocks, Thas's
geometric characterization of flocks \cite{Thas1987} can be reformulated in terms of BLT-sets.

\begin{thm}[\cite{LawPenttila2003}*{Theorem~2.3}]\label{thm_charact}
Let $q$ be an odd prime power. A BLT-set of $\cQ(4,q)$ contained in a hyperplane is either classical or a Kantor semifield BLT-set.
\end{thm}

\subsection{Some preliminary results from group theory}
For positive integers $n$ and $i$, let $C_n$ denote a cyclic group of order $n$. Let $E_q^i$ denote an elementary abelian group of order $q^i$, and write $E_q=E_q^1$. For $g,h\in G$, write $[g,h]:=g^{-1}h^{-1}gh$. The derived subgroup is $G'=[G,G]:=\la[g,h]:g,h\in G\ra$, and $G$ is \emph{perfect} if $G'=G$. We write $G^{(\infty)}$ for the last term of the derived series of $G$. For a linear or semilinear group $H$, $\overline{H}$ denotes its image under the relevant projection map. If $G_1,G_2\leq\SL_2(q)$ both contain $-I_2$, their central product is
\[
G_1\circ G_2:=(G_1\times G_2)/\la(-I_2,-I_2)\ra.
\]

Let $n>1$ and $k\geq 1$ be integers. A divisor $t$ of $n^k-1$ is
\emph{primitive} if it is coprime to $n^i-1$ for $1\leq i<k$. The largest
such divisor is called the \emph{primitive part} $\Phi_k^*(n)$ of $n^k-1$. A prime divisor $t$ of $n^k-1$ is a \emph{primitive prime divisor}, also called a \emph{Zsigmondy prime}, if $n$ has order $k$ modulo $t$.
We use Zsigmondy's theorem in the following form.
\begin{lemma}[\cite{Zsigmondy}]\label{lem_Zsigmondyprime}
{Let $n>1$ and $k\geq3$, with $(n,k)\ne(2,6)$. Then
$n^k-1$ has a primitive prime divisor.}
\end{lemma}

{A primitive prime divisor $t$ satisfies $t\equiv1\pmod{k}$, so
$\gcd(t,k)=1$. Thus, if such a divisor exists, then $\Phi_k^*(n)>1$ and
$\gcd(\Phi_k^*(n),k)=1$. In particular, $n^k-1$ has a primitive prime
divisor when $n$ is odd and $k>2$. The number $n^2-1$ also has a primitive
prime divisor if $n+1$ is not a power of $2$.}

The reducible geometric subgroups of the classical groups in dimension
five were determined by Mitchell \cite{Mitchell1914} (via the isomorphism $\POmega_5(q)\cong\mathrm{PSp}_4(q)$); see also
\cites{Kantor1982Rank3} and \cite{BHR407}*{Table~8.22}. 

\begin{thm}\label{thm_SO_5(q)}
Let $q=p^e$ be an odd prime power with $q\geq5$, and let
$X\leq\Gamma\mathrm O_5(q)$ act reducibly on $V=\F_q^5$. Put $X_{0}:=X\cap\SO_5(q)$. Then $X$ preserves one of the following
geometries, and $X_{\mathrm{0}}$ is a subgroup of the corresponding
stabilizer:
\begin{enumerate}
\item[(i)] $X$ fixes a point of $\cQ(4,q)$, and
$X_{\mathrm{0}}\leq E_q^3:(C_{q-1}\times\SO_3(q))$;
\item[(ii)] $X$ stabilizes a hyperbolic pair of non-collinear points of
$\cQ(4,q)$, and $X_{\mathrm{0}}\leq(C_{q-1}\times\SO_3(q)).2$;
\item[(iii)] $X$ stabilizes a line of $\cQ(4,q)$, and
$X_{\mathrm{0}}\leq U:\GL_2(q)$, where $U$ is a Heisenberg group of order $q^3$,
$Z(U)\cong E_q$, and $U/Z(U)\cong E_q^2$;
\item[(iv)] $X$ stabilizes an external line to $\cQ(4,q)$, and  $X_{\mathrm{0}}\leq(C_{q+1}\times\SO_3(q)).2$;
\item[(v)] $X$ fixes a non-singular point $P$ such that $P^\perp\cap\cQ(4,q)$ is an elliptic quadric, and $X_{\mathrm{0}}\leq\Stab_{\SO_5(q)}(P)\cong\mathrm{O}_4^-(q)$;
\item[(vi)] $X$ fixes a non-singular point $P$ such that $P^\perp\cap\cQ(4,q)$ is a hyperbolic quadric, and $X_{\mathrm{0}}\leq\Stab_{\SO_5(q)}(P)\cong\mathrm{O}_4^+(q)$.
\end{enumerate}
\end{thm}

We will use the following consequence of Dickson's classification; see \cite{Dickson1904}*{Chapter~XII} and \cite{King2005Survey}. Recall that $G^{(\infty)}$ is the last term of the derived series of a group $G$.
\begin{lemma}\label{lem_subPSL}
Let $H$ be a non-solvable subgroup of $\PSL_2(q)$, where
$q=p^e$ is a prime power. Then one of the following holds:
	\begin{enumerate}
		\item[(1)] $H^{(\infty)}\cong \PSL_2(q_0)$ with $q=q_0^m$ for some positive
		integer $m$;
		\item[(2)] $H^{(\infty)}\cong A_5$ and $q\equiv\pm1\pmod{10}$.
	\end{enumerate}
\end{lemma}

\begin{lemma}[\cite{Nelson-thesis}*{Lemma 2.5.2}]\label{Nelson-lemma}
Let $X$ act transitively on a set $\cB$ of size $q+1$, and let
$N\unlhd X$. Then, for every $P\in\cB$, we have $[N:N_P]\mid (q+1)$, where $N_P$ is the stabilizer of $P$ in $N$.
\end{lemma}
\begin{proof}
Since $N\unlhd X$ and $X$ is transitive on $\cB$, all $N$-orbits on
$\cB$ have the same size $[N:N_P]$. Their common size divides
$|\cB|=q+1$.
\end{proof}

By Dickson's classification \cite{Dickson1904}*{Chapter XII}, if $q> 5$, a subgroup
of $\PSL_2(q)$ with index dividing $q+1$ is either $\PSL_2(q)$ or
$E_q:C_{(q-1)/2}$; see also \cite{Huppert} and
\cite{Nelson-thesis}*{pp.~23--24}. Up to conjugacy, the full preimage of
$E_q:C_{(q-1)/2}$ in $\SL_2(q)$ consists of  matrices
\begin{equation}\label{eqn-sub_SL_2(q)}
\begin{pmatrix}
x & 0 \\
y & x^{-1}
\end{pmatrix},
\qquad x\in \F_{q}^{\times},\quad y\in\F_q.
\end{equation}
Such a subgroup is a Borel subgroup of $\SL_2(q)$. In addition, there is a natural homomorphism $\rho:\SL_2(q)\to\Omega_3(q)$ with kernel $\{\pm I_2\}$, which induces $\PSL_2(q)\cong\Omega_3(q)$. For $A=\begin{pmatrix}
	a & b \\
	c & d
	\end{pmatrix}\in\SL_2(q)$, we have
\begin{equation}\label{eqn_Omega3q}
\rho(A)=\begin{pmatrix}
	a^2 & 2ab & b^2\\
	ac & ad+bc & bd \\
	c^2 & 2cd & d^2
	\end{pmatrix}.
\end{equation}
The matrix $\rho(A)$ preserves the form $Q_{\mathrm{std}}(\bfv)=v_1v_3-v_2^2$ for $\bfv=(v_1,v_2,v_3)\in\F_q^3$.

%

\begin{lemma}\label{lem_short_omega3_orbits}
Let $q>5$ be an odd prime power. Define
$Q(x_1,\ldots,x_5)=x_1x_5+x_2x_4-x_3^2$ and set
$L_0:=\{\operatorname{diag}(1,A,1):A\in\Omega_3(q)\}\cong\Omega_3(q)$.
On the points of $\cQ(4,q)$, the only $L_0$-orbits of length at
most $q+1$ are the two fixed points $\la\bfe_1\ra$ and $\la\bfe_5\ra$,
together with the conic
$\{\la(0,\bfw,0)\ra: 0\ne\bfw\in\F_q^3,\, Q_{\mathrm{std}}(\bfw)=0\}$ of length $q+1$.
\end{lemma}

\begin{proof}
Suppose that $\bfv=(x_1,\bfw,x_5)$ is a vector in $V$ with
$\bfw\in\F_q^3$. If $x_1=x_5=0$, the three projective $\Omega_3(q)$-orbits have lengths
$q+1$, $q(q-1)/2$, and $q(q+1)/2$; see
\cite{Hirschfeld1998}*{p.~170} and \cite{KL129}*{Lemma 2.10.5}. The first
orbit is exactly this conic.

Suppose that $(x_1,x_5)\ne(0,0)$. A group element $g\in L_0$ stabilizing the projective
point $\la(x_1,\bfw,x_5)\ra$ must fix $\bfw$ itself. If $\bfw$ is nonzero and singular, its stabilizer in $\Omega_3(q)$ has order $q$, so its orbit has length $(q^2-1)/2$. If $\bfw$ is non-singular, its orbit is at least as long as the corresponding projective orbit, and therefore has length at least $q(q-1)/2$, see \cite{KL129}*{Lemma 2.10.5}. Both bounds exceed $q+1$. Finally, if
$\bfw=0$, the condition $Q(\bfv)=0$ implies that $x_1x_5=0$, yielding precisely the two fixed
points in the statement.
\end{proof}

\subsection{{Quadratic-character sum estimates}}

\begin{lemma}\label{lem_xisquare}
Let $\xi=\omega^{q-1}$, where $\omega$ is a primitive element of $\F_{q^2}$. Then for $1\leq k\leq q$,
	\[
	 \xi^{k}+\xi^{-k}-2=(\xi^{k/2}-\xi^{-k/2})^2\in\begin{cases}
	  \square_q,&\text{if $k$ is odd},\\
	  \blacksquare_q,&\text{if $k$ is even}.
	 \end{cases}
	\]
Here $\xi^{k/2}=\omega^{(q-1)k/2}$.
\end{lemma}
\begin{proof}
Since $\xi=\omega^{q-1}$, we have $\xi^{(q+1)/2}=-1$ and
$\xi^{kq/2}=(-1)^{k}\xi^{-k/2}$. Thus
\[
(\xi^{k/2}-\xi^{-k/2})^q
=(-1)^{k+1}(\xi^{k/2}-\xi^{-k/2}).
\]
Put $u=\xi^{k/2}-\xi^{-k/2}$. Then $u\ne0$ because $1\leq k\leq q$.
If $k$ is odd, then $u^q=u$, so $u\in\F_q$ and $u^2\in\square_q$. If
$k$ is even, then $u^q=-u$, so
$(u^2)^{(q-1)/2}=u^{q-1}=-1$ and $u^2\in\blacksquare_q$.
\end{proof}

We recall the following Weil bound, which will be used in Section \ref{sec:4}.

\begin{thm}[ \cite{Finitefield}*{Theorem 5.41}]\label{thm_FFquadsum}
{Let $\psi: \F_{q}^{\times}\to\mathbb C$ be a multiplicative character of order $m$, extended to $\F_q$  by setting $\psi(0)=0$. Let $f\in\F_q[x]$ be monic and not an $m$-th power, and let $d$ be the number of distinct roots of $f$
in its splitting field. Then, for every $a\in\F_q$,}
\[
\left|\sum_{x\in\F_q}\psi(af(x))\right|\leq(d-1)\sqrt q.
\]
	\end{thm}

\begin{lemma}\label{lem_caseI_torus_sum}
{Let $\xi=\omega^{q-1}$ be as in Lemma \ref{lem_xisquare}, let $k$ be an
even divisor of $q+1$, and put  $Y:=\big\la\xi^k\big\ra$. Suppose that
$\alpha\in\F_{q^2}^{\times}$, $\beta\in\F_q$, and
$\beta^2-4\alpha\alpha^q\ne0$. Then
\[
\left|
\sum_{y\in Y}
\chi\bigl(\beta+\alpha y+\alpha^qy^{-1}\bigr)
\right|
\leq
{\frac{(2k-1)\sqrt q+1}{k}}
<
2\sqrt q.
\]}
\end{lemma}

\begin{proof}
{The group $\langle\xi\rangle$ has order $q+1$.
For each $y\in\langle\xi\rangle$, we have $y^q=y^{-1}$, so
$\beta+\alpha y+\alpha^qy^{-1}\in\F_q$.} The map $\big\la\xi\big\ra\to Y$,
$x\mapsto x^k$, is $k$-to-one. Hence
\[
\sum_{y\in Y}\chi\bigl(\beta+\alpha y+\alpha^qy^{-1}\bigr)
=
\frac{1}{k}
\sum_{x\in\big\la\xi \big\ra}
\chi\bigl(\beta+\alpha x^k+\alpha^qx^{-k}\bigr).
\]

Choose $\delta\in\F_{q^2}^{\times}$ with
$\delta^q=-\delta$ and $\delta^2\in\blacksquare_q$. The map
\[
x(t)=\frac{t+\delta}{t-\delta},
\qquad t\in\F_q\cup\{\infty\},
\]
is a bijection onto $\big\la\xi\big\ra$, with $t=\infty$ corresponding to $x=1$.
Moreover,
$x(t)^q=x(t)^{-1}$ for all $t\in\F_q$. For each $t\in\F_q$, we have
\[
\beta+\alpha x(t)^k+\alpha^qx(t)^{-k}
=\frac{f(t)}{(t^2-\delta^2)^k},
\]
where
$f(t):=\alpha(t+\delta)^{2k}+\beta(t^2-\delta^2)^k
+\alpha^q(t-\delta)^{2k}$. Expanding the first and third terms gives
$\alpha(t+\delta)^{2k}+\alpha^q(t-\delta)^{2k}
=\sum_{i=0}^{2k}\binom{2k}{i}
\bigl(\alpha\delta^{2k-i}+\alpha^q(-\delta)^{2k-i}\bigr)t^i$,
and every coefficient lies in $\F_q$. Hence $f(t)\in\F_q[t]$. Since
$\delta^2\in \blacksquare_q$, $t^2-\delta^2\ne0$ for
$t\in\F_q$. Since $k$ is even, $(t^2-\delta^2)^k$ is a square. Therefore
$\chi(\beta+\alpha x(t)^k+\alpha^qx(t)^{-k})=\chi(f(t))$.

To use Weil's bound, it remains to determine the number of distinct roots of $f(t)$. Away from
$t=\pm\delta$, put $U=x(t)^k$. Then $f(t)=0$ is equivalent to
$\alpha U^2+\beta U+\alpha^q=0$. By hypothesis, its discriminant
$\beta^2-4\alpha^{q+1}$ is nonzero, so the equation has two distinct
nonzero roots. Since $k\mid (q+1)$, we have $p\nmid k$, and each root
therefore gives $k$ distinct solutions of $x(t)^k=U$.  Neither
$t=\delta$ nor $t=-\delta$ is a root of $f$. The point
$t=\infty$ corresponds to one of these $2k$ solutions precisely
when $\alpha+\beta+\alpha^q=0$. Thus the number $d$ of distinct finite roots of
$f$ is $2k$ if $\alpha+\beta+\alpha^q\ne0$, and $2k-1$ otherwise.
All these roots are simple, so $f$ is not a square in $\F_q[t]$.
After multiplying $f$ by a nonzero element of $\F_q$ to make it monic, Theorem \ref{thm_FFquadsum} gives
\[
\left|\sum_{t\in\F_q}\chi(f(t))\right|
\leq(d-1)\sqrt q\leq(2k-1)\sqrt q.
\]
The point $t=\infty$ contributes
$\chi(\alpha+\beta+\alpha^q)$, whose absolute value is at most $1$.
Consequently,
\[
\left|
\sum_{x\in\big\la\xi \big\ra}\chi\bigl(\beta+\alpha x^k+\alpha^qx^{-k}\bigr)
\right|
\leq(2k-1)\sqrt q+1.
\]
Dividing by $k$ proves the first bound. Finally,
$2\sqrt q-((2k-1)\sqrt q+1)/k=(\sqrt q-1)/k>0$, which gives the desired
inequality.
\end{proof}
\subsection{Representations of \texorpdfstring{$A_5$}{A5} and
\texorpdfstring{$2.A_5$}{2.A5}}\label{sec:2_A5}
We record explicit absolutely irreducible representations of $A_5$ and $2.A_5$. 
Assume that $q\equiv\pm1\pmod{10}$. Let $z$ be a primitive fifth root
of unity in $\overline{\F}_q$. Since $5\mid q^2-1$, we have
$z\in\F_{q^2}$. Put $b_5=z+z^4$. The congruence condition on $q$ implies that $b_5^q=b_5$, and hence $b_5\in\F_q$. Moreover,
$b_5^2+b_5-1=0$ and $z^2+z^3=-1-b_5$.
\medskip

\noindent (1) Suppose in addition that $q\equiv-1\pmod{10}$. Then
$z^q=z^{-1}=z^4$. We use the following realization of
$\SL_2(q)$ as $\SU_2(q)$:
\[
\SU_2(q)=\left\{
\begin{pmatrix}\alpha&\beta\\-\beta^q&\alpha^q\end{pmatrix}:
\alpha,\beta\in\F_{q^2},\quad
\alpha^{q+1}+\beta^{q+1}=1
\right\}.
\]
Define
\begin{equation}\label{eqn_2A5}
A_{\mathrm u}:=\begin{pmatrix}
	(z-z^4)^{-1} & (z^2-z^3)^{-1}\\
	(z^2-z^3)^{-1} & (z^4-z)^{-1}
\end{pmatrix},
\qquad
B_{\mathrm u}:=\begin{pmatrix}
z & 0 \\
0 &z^4
\end{pmatrix}.
\end{equation}
The identity $\sum_{i=1}^{4}z^i=-1$ gives
$\det(A_{\mathrm u})=\det(B_{\mathrm u})=1$,
$A_{\mathrm u}^2=-I_2$, $B_{\mathrm u}^5=I_2$, and
$(A_{\mathrm u}B_{\mathrm u})^3=-I_2$.
Hence
$\la A_{\mathrm u},B_{\mathrm u}\ra\cong 2.A_5\leq\SU_2(q)\cong\SL_2(q)$.\medskip

\noindent(2) Since $b_5^2+b_5-1=0$, the equation $u^2+u+v^2-b_5v+1=0$ is equivalent to $(u+\frac{1}{2})^2+(v-\frac{b_5}{2})^2=-\frac{1}{2}-\frac{b_5}{4}$. By Theorem \ref{thm_FFquadsum}, the number of solutions of this equation is at least $q-\sqrt{q}>0$. Choose a suitable pair $(u,v)\in\F_q^2$, and define
\[
A_{\mathrm s}:=\begin{pmatrix}0&-1\\1&0\end{pmatrix},
\qquad
B_{\mathrm s}:=\begin{pmatrix}u&v\\v-b_5&-u-1\end{pmatrix}.
\]
A direct computation shows that  $A_{\mathrm s}^2=-I_2$,
$B_{\mathrm s}^3=I_2$, and $(A_{\mathrm s}B_{\mathrm s})^5=I_2$. Hence
$\la A_{\mathrm s},B_{\mathrm s}\ra\cong 2.A_5\leq\SL_2(q)$.\medskip

\noindent{(3) Let}
	\[ A_3:= \begin{pmatrix}
-1 & 0 & 1+b_5 \\
0 & -1 & 1+b_5 \\
0 & 0 &1
\end{pmatrix},
\qquad B_3:=\begin{pmatrix}
0 &0 & 1\\
1 & 0 & 0\\
0 & 1 &0
\end{pmatrix},
\qquad A_3B_3=\begin{pmatrix}
0 & 1+b_5  & -1 \\
-1 & 1+b_5 &  0 \\
0 & 1 & 0
\end{pmatrix}.\]
Here $A_3^2=B_3^3=(A_3B_3)^5=I_3$, so
$\la A_3,B_3\ra\cong A_5\leq\SL_3(q)$. This representation preserves the
symmetric bilinear form with Gram matrix
\[
\begin{pmatrix}
-1-2b_5 & 1 & 1\\
1 & -1-2b_5 & 1\\
1 & 1 & -1-2b_5
\end{pmatrix}.
\]
The corresponding quadratic form on $\F_q^3$ is
\begin{equation}\label{eqn_QA5}
Q_{A_5}(\bfv)=v_1v_2+v_1v_3+v_2v_3
-\left(\frac12+b_5\right)(v_1^2+v_2^2+v_3^2),
\qquad \bfv=(v_1,v_2,v_3)\in\F_q^3.
\end{equation}
Its associated bilinear form is
$B_{A_5}(\bfv,\bfw)=Q_{A_5}(\bfv+\bfw)-Q_{A_5}(\bfv)-Q_{A_5}(\bfw)$.
The subgroups isomorphic to $A_5$ form two conjugacy classes in
$\Omega_3(q)$, and these classes are fused in $\SO_3(q)$.
Thus it suffices to fix one representative for the arguments below;
see \cite{BHR407}*{Section~4, Table~4.4}.

\section{The \texorpdfstring{$\PSL_2(q_0)$}{PSL}-case}\label{sec:3}
Throughout this section, assume that $q>5$ and that $\overline G$ is reducible, non-solvable, and
transitive on the BLT-set $\cB$. By \eqref{eqn_orthogonal_core}, the quotient $\overline G/\pi(\Gso)$ is solvable, so $\Gso$ is non-solvable.  
By Theorem \ref{thm_SO_5(q)} and Lemma \ref{lem_subPSL}, we split the discussion into two cases:
\begin{itemize}
    \item[(1)]  $\Gso$ has at least one non-abelian composition factor isomorphic to $\PSL_2(q_0)$, where $q_0$ is a power of $p$;
\item[(2)] All non-abelian composition factors of $G$ are $A_5$, where $p\ne 5$.
\end{itemize}
 This section treats the former case \big(i.e., Case \textup{(1)}\big), and Section \ref{sec:4} treats the latter one. We refer to these two conditions as the $\PSL_2(q_0)$-condition and $A_5$-condition in the sequel. The transitivity condition implies $(q+1) \mid|\overline G|$, and
\eqref{eqn_orthogonal_core} therefore yields
\begin{equation}\label{eqn_divisiblecond}
 \frac{q+1}{\gcd(e,q+1)}\bigm| |\Gso|.
\end{equation}

\begin{lemma}\label{lem_primitive}
Let $q=p^e$. If $e=1$, then $(q+1)\mid|\Gso|$. If $e>1$, then $\Phi_{2e}^*(p)\mid|\Gso|$. In the latter case (i.e., $e >1$), for every prime divisor
$r$ of $\Phi_{2e}^*(p)$, one may choose a composition factor $S$ of
$\Gso$ such that $r\mid|S|$; if this factor is isomorphic to
$\PSL_2(p^t)$, then $e\mid t$.
\end{lemma}

\begin{proof}
The assertion in the $e=1$ case follows immediately from
\eqref{eqn_divisiblecond}. Suppose that $e>1$. By Lemma
\ref{lem_Zsigmondyprime}, $\Phi_{2e}^*(p)>1$. Moreover,
$\Phi_{2e}^*(p)$ divides $p^e+1=q+1$ and is coprime to $e$. Hence
$\Phi_{2e}^*(p)$ divides $(q+1)/\gcd(e,q+1)$, and so
$\Phi_{2e}^*(p)\mid|\Gso|$ by \eqref{eqn_divisiblecond}.

Let $r$ be a prime divisor of $\Phi_{2e}^*(p)$. Since the product of the
orders of the composition factors of $\Gso$ is $|\Gso|$, some composition
factor $S$ satisfies $r\mid|S|$. Such an $r$ is a primitive prime divisor
of $p^{2e}-1$, so the order of $p$ modulo $r$ is $2e$. If
$S\cong\PSL_2(p^t)$, then $r\ne p$ and the order formula for
$\PSL_2(p^t)$ gives $r\mid p^{2t}-1$. Therefore $2e\mid2t$, and hence
$e\mid t$.
\end{proof}

\begin{lemma}\label{lem_unique_nonabelian_factor}
Let $\cB$ be a BLT-set of $\cQ(4,q)$, where $q>5$ is odd, and suppose that $\overline G=\Aut(\cB)$ is reducible, non-solvable and acts transitively on $\cB$. 
\begin{itemize}
\item[(1)] If $\Gso$ has a composition factor isomorphic to $\PSL_2(q_0)$, where $q_0$ is a power of $p$, then it is the unique non-abelian composition factor of $\Gso$. 
\item[(2)] If $\Gso$ has two non-abelian composition factors, then both are isomorphic to $A_5$, and $\Gso^{(\infty)}\cong 2.A_5\circ 2.A_5$.
\end{itemize}
\end{lemma}
\begin{proof} 
By the structure of $\Gso$ in Theorem \ref{thm_SO_5(q)} together with Lemma \ref{lem_subPSL}, $\Gso$ has at most two non-abelian composition factors. Suppose that $\Gso$ has two such non-abelian factors $H_1$ and $H_2$, and at least one of $H_1,H_2$ is isomorphic to $\PSL_2(q_0)$. This can only occur in the
hyperbolic case \textup{(vi)} of Theorem \ref{thm_SO_5(q)}. We therefore use the model
\[
V=\F_q\oplus M_2(q),\qquad Q(x,M)=x^2-\det(M).
\]
In this model, $\Gso^{(\infty)}=\widehat H_1\circ\widehat H_2$, and
$\widehat H_i/Z(\widehat H_i)\cong H_i$, where the two quasisimple subgroups
lie in the two $\SL_2(q)$-factors, respectively.

After interchanging the two factors $H_1,H_2$ if necessary, we can assume that $H_1\cong\PSL_2(q_0)$. The corresponding quasisimple subgroup $\widehat H_1\cong\SL_2(q_0)$ acts by left multiplication. Let $P=\la(x,M)\ra\in\cB$. If $x\ne0$, then $M$ is invertible, and no nonidentity unipotent $u\in\SL_2(q_0)$ stabilizes $P$. Choose such an
element with $\tr(u)=2$. Then
\[
B\bigl((x,M),(x,uM)\bigr)=(2-\tr(u))x^2=0,
\]
which contradicts the BLT property. On the other hand, if $x=0$, then $M$ has rank one.
Since $\SL_2(q_0)$ acts absolutely irreducibly on its natural
two-dimensional module, there exists $u$ such that
$\la uM\ra\ne\la M\ra$. Therefore
$\det(M+uM)=\det(I+u)\det(M)=0$ and
$B\bigl((0,M),(0,uM)\bigr)=0$, which is also a
contradiction. The same argument applies after interchanging the two $\SL_2(q)$-factors. Hence, neither $H_1$
nor $H_2$ can be isomorphic to $\PSL_2(q_0)$. By Lemma
\ref{lem_subPSL}, both $H_1$ and $H_2$ are isomorphic to  $A_5$,  and their quasisimple covers in $\SL_2(q)$ are $2.A_5$. Therefore $\Gso^{(\infty)}\cong2.A_5\circ2.A_5$. This proves the second assertion $(2)$.
\end{proof}

\subsection{The reducible geometric cases under the
\texorpdfstring{$\PSL_2(q_0)$}{PSL}-condition}
The cases $q\leq29$ and $31\leq q\leq67$ are classified in
\cite{LawPenttila2003} and \cite{Betten2013}*{Theorem~1 and Table~1},
respectively; coordinates and stabilizers are available in
\cite{SmallBLT}. Hence assume $q\geq71$ and that
$\Gso$ has a composition factor $\PSL_2(q_0)$. This is its unique non-abelian factor by Lemma \ref{lem_unique_nonabelian_factor}. For each geometry, we use the stabilizer in $\SO_5(q)$ listed in Theorem \ref{thm_SO_5(q)}.

Whenever $e>1$ and the stabilizer forces a primitive divisor supplied by Lemma \ref{lem_primitive} into this factor, writing it as
$\PSL_2(p^t)$ gives $e\mid t$; the same conclusion is automatic for $e=1$. Thus $t=e$ if $t\mid e$, while $t\in\{e,2e\}$ if $t\mid2e$. We do not use this argument in the external-line case, whose stabilizer contains a cyclic factor of order $q+1$.

We will use the following model for the cases \textup{(i)--(iii)} of Theorem
\ref{thm_SO_5(q)}. Let $V=\F_q^5$, with standard basis
$\{\bfe_1,\ldots,\bfe_5\}$, and equip $V$ with the quadratic form
\begin{equation*}\label{eqn_Q4qmodel}
Q(\bfx)=x_1x_5+x_2x_4-x_3^2,
\qquad
B(\bfx,\bfy)=Q(\bfx+\bfy)-Q(\bfx)-Q(\bfy).
\end{equation*}
With respect to the standard basis, we have $\operatorname{disc}(Q)=-2$.


\begin{proposition}\label{prop_case1}
{Suppose that $G$ stabilizes a point of $\cQ(4,q)$.
If the $\PSL_2(q_0)$-condition holds and $q\geq 71$ is odd, then no transitive
BLT-sets $\cB$ exist.}
\end{proposition}

\begin{proof}
Suppose that such a transitive BLT-set $\cB$ exists. Since $\operatorname{PSO}_5(q)$ acts transitively on the points of $\cQ(4,q)$, we assume that
$P=\la\bfe_1\ra$ without loss of generality. Let $K:=\Stab_{\SO_5(q)}(P)\cong E_q^3:(C_{q-1}\times \SO_3(q))$. Then
\[
 K^{(\infty)}=U:L_0,
 \qquad
 L_0:=\{\operatorname{diag}(1,A,1):A\in\Omega_3(q)\},
\]
where $U$ is the unipotent radical isomorphic to $E_q^3$ and $L_0$ is the
Levi factor. It follows that $\Gso\leq K$. Set $X:=\Gso/(\Gso\cap U) \cong \Gso U/U \leq K/U$. Since $\Gso\cap U$ is a solvable $p$-group and $\Gso^{(\infty)}\leq K^{(\infty)}$, we have $X^{(\infty)}\leq L_0U/U$. By the $\PSL_2(q_0)$-condition and Lemma \ref{lem_subPSL}, we obtain
$X^{(\infty)}\cong\Omega_3(p^t)$ for some $t\mid e$. Observe that $\Gso\cap U$ is an elementary abelian $p$-group and every prime divisor of $[X:X^{(\infty)}]$ divides $2(q-1)$. Hence if $e>1$, $\Phi_{2e}^*(p)>1$ by Lemma \ref{lem_Zsigmondyprime}. It follows from Lemma \ref{lem_primitive} that $\Phi_{2e}^*(p)$ divides $|X^{(\infty)}|$, and then $e=t$. Thus,
\[
 X^{(\infty)}=L_0U/U\cong L_0\cong\Omega_3(q).
\]
Consider the natural projection $\rho:\ K^{(\infty)}\to L_0$. Then $\rho(\Gso^{(\infty)})=L_0$. Put
$N:=\Gso^{(\infty)}\cap U$. The natural action of $L_0\cong \Omega_3(q)$ on $U$ by conjugation is irreducible over $\F_p$. Hence, the only $L_0$-invariant subgroups are $1$ and $U$, i.e.,  either $N=1$
or $N=U$. If $N=U$, then
$\Gso^{(\infty)}=U:L_0$. If $\Gso^{(\infty)}\cap U=1$, then $\Gso^{(\infty)}$ is a complement to $U$ in $U:L_0$, and
\[
 H^1(L_0,U)=H^1(\Omega_3(q),\F_q^3)=0
\]
by \cite{ClineParshallScott1975}*{Table~4.5}. Hence all such complements are conjugate to $L_0$ by an element of $U$. Without loss of generality, we may assume that $L_0\leq \Gso^{(\infty)}$ in both cases.

Since $\cB$ is $L_0$-invariant and $|\cB|=q+1$, Lemma
\ref{lem_short_omega3_orbits} gives
\[
\cB=\{\la(0,\bfw,0)\ra: {\bf 0}\ne \bfw\in \F_q^3,\,\, Q_{\mathrm{std}}(\bfw)=0\}.
\]
Every point of this conic is collinear with $\la\bfe_1\ra$, contradicting
the BLT property. Thus no such transitive BLT-sets exist.
\end{proof}

\begin{proposition}\label{prop_case2}
Suppose that $G$ stabilizes a pair of non-collinear points of $\cQ(4,q)$. If the
$\PSL_2(q_0)$-condition holds and $q\geq 71$ is odd, then no transitive BLT-sets $\cB$ exist.
\end{proposition}

\begin{proof}
Suppose that such a transitive BLT-set $\cB$ exists. By
Witt's theorem \cite{KL129}*{Proposition~2.1.6}, 
$\operatorname{PSO}_5(q)$ acts transitively on all pairs of non-collinear points of $\cQ(4,q)$. Without loss of generality, we therefore assume that the pair of points stabilized by $G$ is $\{P,Q\}=\{\la\bfe_1\ra,\la\bfe_5\ra\}$.

Let $W:=\la\bfe_1,\bfe_5\ra^\perp$. Then $W$ is a three-dimensional
non-degenerate orthogonal space. By the structure of the stabilizer of $\{P,Q\}$ given in Theorem \ref{thm_SO_5(q)}, $\Gso^{(\infty)}\cong \Omega_3(p^t)$ for some $t\mid e$, and it induces a natural action on $W$. 
Every abelian composition factor of $\Gso$ has order dividing
$2(q-1)$. Hence the preceding primitive-divisor argument, together
with Lemma \ref{lem_primitive}, gives $t=e$. Consequently,
$\Gso^{(\infty)}$ induces $\Omega_3(q)$ acting naturally on $W$, and acts trivially on $\la\bfe_1,\bfe_5\ra$. Thus
$L_0:=\{\operatorname{diag}(1,A,1):A\in\Omega_3(q)\}
\leq\Gso^{(\infty)}$.
 Hence, by Lemma
\ref{lem_short_omega3_orbits},
\[
 \cB=\PG(W)\cap\cQ(4,q).
\]
Every point of this conic is collinear with both $P$ and $Q$, contradicting
the BLT property.
\end{proof}
\begin{proposition}\label{prop_case3}
Suppose that $G$ stabilizes a line of $\cQ(4,q)$. If the
$\PSL_2(q_0)$-condition holds and $q\geq 71$ is odd, then no such
transitive BLT-sets $\cB$ exist.
\end{proposition}

\begin{proof}
Suppose that such a transitive BLT-set $\cB$ exists. By
Witt's theorem \cite{KL129}*{Proposition 2.1.6},
 $\mathrm{PSO}_5(q)$ acts transitively on the set of all totally singular
 lines. Up to conjugacy, we assume that the
fixed line is
\[
\ell=\{\la(a,b,0,0,0)\ra:\ a,b\in\F_q,\ (a,b)\ne(0,0)\}.
\]
Put $K:=\Stab_{\SO_5(q)}(\ell)$. Then
$K=U:\GL_2(q)$, where $U$ is the Heisenberg group of order
$q^3$. 
Its center is
$Z:=Z(U)\cong E_q$, and $U/Z$ is the natural
two-dimensional $\SL_2(q)$-module. The Levi factor $\GL_2(q)$ is represented by matrices of the form
\[
M_A:=
\begin{pmatrix}
A & &\\
& 1 &\\
& & JA^{-\top}J
\end{pmatrix},
\qquad
J=\begin{pmatrix}0&1\\1&0\end{pmatrix},
\]
where $A\in\GL_2(q)$.

Let $H:=\Gso^{(\infty)}$, and let $\rho:K\to\GL_2(q)$ be the Levi
projection. By Lemma \ref{lem_subPSL} and the $\PSL_2(q_0)$-condition, $\big(\rho(\Gso)\big)^{(\infty)}=\rho(H)$ is conjugate to $\SL_2(p^t)$ for some $t\mid e$. By the structure of $K$ in Theorem \ref{thm_SO_5(q)}, every abelian composition factor of $\Gso$ has order dividing $2p(q-1)$. When $e>1$, by the preceding argument on primitive prime divisors of $q^2-1$ and Lemma \ref{lem_primitive}, we have $e=t$. Consequently,
$\rho(H)=\SL_2(q)$. Set $L_0:=\{M_A:A\in\SL_2(q)\}$.

Write $N:=H\cap U$.  Since $N\unlhd H$, the subgroup
$NZ/Z$ is an $L_0$-invariant additive subgroup of $U/Z$. Since $L_0\cong
\SL_2(q)$ is transitive on the nonzero vectors of its natural module
$U/Z$, we have $NZ/Z=1$ or $NZ/Z=U/Z$.
If $NZ=U$, then, by $Z=[U,U]$, we have 
\[
 Z=[NZ,NZ]=[N,N]\leq N.
\]
Thus $N=U$ and $H=U:L_0$. On the other hand, suppose that $N\leq Z$. Then $HZ/Z$ is a complement to $U/Z$ in $(U/Z):L_0$. By \cite{ClineParshallScott1975}*{Table~4.5},
\[
 H^1(L_0,U/Z)=H^1(\SL_2(q),\F_q^2)=0.
\]
Thus this complement is conjugate by an element of $U/Z$ to the standard
complement $L_0Z/Z$. Hence, up to conjugation, we may assume that $H\leq ZL_0$. The group $L_0$ centralizes $Z$, because its action on $Z$ is given by the determinant. Since $H$ is perfect, we have 
\[
 H=[H,H]
 \leq[ZL_0,ZL_0]=L_0.
\]
Since $\rho(H)=\SL_2(q)$, we have $H=L_0$. Thus, after conjugating by an
element of $U$ when necessary, $L_0\leq H$ in both cases, and $\cB$ is a
union of $L_0$-orbits.

Fix $P_\infty=\la\bfx\ra\in\cB$, with
$\bfx=(x_1,x_2,x_3,x_4,x_5)^\top$, and let
$A=\begin{pmatrix}a&b\\ c&d\end{pmatrix}\in\SL_2(q)$. Then
\[
 M_A\bfx=
(a x_1+b x_2,\ c x_1+d x_2,\ x_3,\ a x_4-b x_5,\ -c x_4+d x_5)^\top.
\]
Using $Q(\bfx)=0$, a direct calculation gives
\[
 B(\bfx,M_A\bfx)=(\tr(A)-2)x_3^2.
\]
Take
\[
 A_+=\begin{pmatrix}1&1\\0&1\end{pmatrix},\qquad
 A_-=\begin{pmatrix}1&0\\1&1\end{pmatrix}.
\]
Both are nonidentity elements of trace $2$. If $x_3\ne0$, then
$x_1x_5+x_2x_4=x_3^2\ne0$, so $(x_2,x_5)\ne(0,0)$ and
$\la\bfx\ra\ne\la M_{A_+}\bfx\ra$. If $x_3=0$ and
$(x_2,x_5)\ne(0,0)$, the same
conclusion holds. Finally, if $x_3=x_2=x_5=0$, then
$(x_1,x_4)\ne(0,0)$ and $\la\bfx\ra\ne\la M_{A_-}\bfx\ra$. In every case,
$\cB$ contains two distinct collinear points, contradicting the BLT
property.
\end{proof}

\begin{proposition}\label{prop_case4}
Suppose that $G$ stabilizes an external line $\ell$ to $\cQ(4,q)$ and its orthogonal complement $\ell^\perp$. If the $\PSL_2(q_0)$-condition holds and $q\geq71$ is odd,
then $\cB$ is classical.
\end{proposition}

\begin{proof}

Suppose that such a transitive BLT-set $\cB$ exists. We use the model
\[
 V=\F_{q^2}\oplus\F_q^3,\qquad
 Q(\alpha,(y_1,y_2,y_3))=\alpha^{q+1}-y_1y_3+y_2^2,
\]
in which
\[
 \ell=\{\la(\alpha,0)\ra:\alpha\in\F_{q^2}^*\},\qquad
 \ell^\perp=\PG(\mathbf{0}\oplus\F_q^3).
\]
By Theorem \ref{thm_SO_5(q)}, the stabilizer of
$\ell$ in $\SO_5(q)$ has structure
$(C_{q+1}\times\SO_3(q)).2$. Its perfect residual is the natural
$\Omega_3(q)$-factor on the second summand. Hence
$\Gso^{(\infty)}\leq\Omega_3(q)$.

Put $L:=\Gso^{(\infty)}\cong\Omega_3(q_0)$, where
$q=q_0^m$. The natural three-dimensional $\Omega_3(q_0)$-module is absolutely irreducible; see
\cite{KL129}*{Proposition~2.10.6(i)}. Viewing $\F_q^3$ as the scalar extension
of $\F_{q_0}^3$, the group $\Omega_3(q_0)$ acts irreducibly on $\F_q^3$. Since $L$ is perfect, every one-dimensional $L$-invariant $\F_q$-subspace of $V$ is fixed pointwise and is contained in the first summand. Thus the only projective points fixed by $L$ lie
on the external line $\ell$, but no point of $\ell$ is singular. Hence every
point of $\cB$ has a nonzero component in the second summand $\F_q^3$.

Since $L\unlhd G$ and $G$ is transitive on $\cB$, Lemma
\ref{Nelson-lemma} gives
$[L:L_R]\mid (q+1)$ for every $R\in\cB$. Moreover,
\[
 \gcd(|L|,q+1)=
 \begin{cases}
  q_0+1,&\text{if $m$ is odd},\\
  2,&\text{if $m$ is even}.
 \end{cases}
\]
If $m$ is even, then $[L:L_R]\leq2$. Since $L$ is perfect, this
forces $L_R=L$, contradicting the preceding fixed-point statement. Hence
$m$ is odd and $[L:L_R]\mid q_0+1$. Since $L_R\ne L$, Dickson's
classification \cite{Dickson1904} shows that $L_R$ is a Borel subgroup isomorphic to $E_{q_0}:C_{(q_0-1)/2}$.

Using the model of $\Omega_3(q_0)$ in
\eqref{eqn_Omega3q}, a direct calculation shows that the unique singular point fixed by a standard Borel subgroup of $\Omega_3(q_0)$ is
\[
 X_0=\la(0,(0,0,1))\ra\in\ell^\perp.
\]
After conjugating in $L$, we may take $R=X_0$. Thus
$\cB\cap\ell^\perp\ne\varnothing$. Since $G$ stabilizes
$\ell^\perp$ and is transitive on $\cB$,
\[
 \cB\subseteq\ell^\perp\cap\cQ(4,q).
\]
The conic $\ell^\perp\cap\cQ(4,q)$ has $q+1$ points. Hence equality holds,
and $\cB$ is classical.
\end{proof}

\begin{proposition}\label{prop_case5}
	Suppose that $G$ stabilizes a non-singular point $P$ with $P^\perp\cap\cQ(4,q)\cong Q^-(3,q)$. If $q\geq71$ and the $\PSL_2(q_0)$-condition holds, then $\cB$ is classical.
\end{proposition}

\begin{proof}
Suppose that such a transitive BLT-set $\cB$ exists. Let $V_1$ be the four-dimensional $\F_q$-space of Hermitian $2\times2$ matrices over $\F_{q^2}$, i.e., $V_1=\left\{
 \begin{pmatrix}a&\alpha\\ \alpha^q&b\end{pmatrix}:
 a,b\in\F_q,\ \alpha\in\F_{q^2}
 \right\}$. Use the model
\[
 V=\F_q\oplus V_1,
\]
and write $A^{(q)}$ for the matrix obtained by raising each entry of $A$ to the $q$-th power. Take
\[
 P=\la(1,0)\ra,\qquad Q(x,M)=x^2-\det(M).
\]
Thus $P^\perp=\PG(V_1)$ and $P^\perp\cap\cQ(4,q)\cong Q^-(3,q)$. Moreover,
$\Stab_{\SO_5(q)}(P)\cong\mathrm O_4^-(q)$. The preimage
$2.\Omega_4^-(q)\cong\SL_2(q^2)$ acts on $V_1$ by
$M\longmapsto AM(A^{(q)})^{\top}$ for $A\in\SL_2(q^2)$.

By Lemma \ref{lem_subPSL} together with $\Omega_4^-(q)\cong\PSL_2(q^2)$, we obtain $L:=\Gso^{(\infty)}\cong\PSL_2(p^t)$ for some $t\mid2e$. Since $L$ is characteristic in $\Gso$, we have
$\Gso\leq N_{\mathrm O_4^-(q)}(L)$. Then $N_{\mathrm{O}_4^-(q)}(L) / L$ is a $2$-group, cf. \cites{Dickson1904,BHR407}. Thus all abelian composition factors of $\Gso$
have order $2$. By Lemma \ref{lem_primitive}, we obtain $e\mid t$ (automatically when $e=1$). Since $t\mid2e$, we have $t\in\{e,2e\}$. Consequently, we have
\begin{itemize}
	\item[(i)] $ L\cong\PSL_2(q)$ or
	\item[(ii)] $L=\Omega_4^-(q)\cong\PSL_2(q^2)$.
\end{itemize}
 We now exclude case \textup{(ii)}. Since $L\unlhd G$, every
$L$-orbit on $\cB$ has length dividing $q+1$. The group $\Omega_4^-(q)$ fixes no points of $\cQ(4,q)$, while the two $\Omega_4^-(q)$-orbits have respective sizes $q^2+1$ and $q(q^2+1)$; see \cite{KL129}*{Lemma 2.10.5}. Both lengths exceed $q+1$,  so the case \textup{(ii)} is impossible. Thus $L\cong\PSL_2(q)\cong\Omega_3(q)$.

There are two conjugacy classes of $\PSL_2(q)$ in $\PSL_2(q^2)$; see
\cites{Dickson1904,BHR407}. Let $\omega$ be a primitive element
of $\F_{q^2}$ and set
\[
 D_\eta=\begin{pmatrix}\eta&0\\0&1\end{pmatrix},\qquad
 \eta\in\{1,\omega\}.
\]
Their full preimages in $\SL_2(q^2)$ are represented by $\{D_\eta A D_\eta^{-1}:A\in\SL_2(q)\}$, where $\eta\in\{1,\omega\}$. After conjugating $L$ in $\Omega_4^-(q)$, choose $\eta$ so that this subgroup is the full preimage of $L$. Choose $\delta\in\F_{q^2}^{\times}$ with $\delta^q=-\delta$ and put
\[
 J_\eta:=\begin{pmatrix}0&\eta\delta\\-\eta^q\delta&0\end{pmatrix}.
\]
A direct calculation gives
\[
 C_V(L)=\la(1,0),(0,J_\eta)\ra_{\F_q}.
\]
Set $\ell_\eta:=\PG(C_V(L))$. Since
$L=\Gso^{(\infty)}\unlhd G$, both $\ell_\eta$ and
$\ell_\eta^\perp$ are $G$-invariant. On $C_V(L)$,
\[
 Q(x,cJ_\eta)=x^2-c^2\eta^{q+1}\delta^2,
 \qquad \eta^{q+1}\delta^2\in\F_q^\times.
\]
Now $\delta^2\in\blacksquare_q$, since
$(\delta^2)^{(q-1)/2}=\delta^{q-1}=-1$, and
$\omega^{q+1}\in\blacksquare_q$. Thus
$\eta^{q+1}\delta^2$ is a nonsquare for $\eta=1$ and a square for
$\eta=\omega$. In the first case, $C_V(L)$ is anisotropic, so
$\ell_\eta$ is external, and then $\cB$
is classical by Proposition \ref{prop_case4}. In the second case, $\ell_\eta$ is a secant line whose two singular points are non-collinear because the restriction of $Q$ is
non-degenerate. The group $G$ stabilizes this pair, and Proposition
\ref{prop_case2} rules out the case. Therefore $\cB$ is classical.
\end{proof}

\begin{proposition}\label{prop_case6}
	Suppose that $G$ stabilizes a non-singular point $P$ with $P^\perp\cap\cQ(4,q)\cong Q^+(3,q)$. If $q\geq71$ is odd and the $\PSL_2(q_0)$-condition holds, then $\cB$ is either classical
or a Kantor semifield BLT-set.
\end{proposition}
\begin{proof}

Suppose that such a transitive BLT-set $\cB$ exists. We use the model
\[
 V=\F_q\oplus M_2(q),\qquad P=\la(1,\begin{pmatrix}
 	0 &0 \\ 0 &0
 \end{pmatrix})\ra,\qquad
 Q(x,M)=x^2-\det(M).
\]
Then $\Stab_{\SO_5(q)}(P)$ is isomorphic
to $\mathrm O_4^+(q)$ by Theorem \ref{thm_SO_5(q)}. Its perfect residual is $\Omega_4^+(q)\cong\SL_2(q)\circ\SL_2(q)$, which acts on $V$ by
\[
 (A\circ C):\, (x,M)\mapsto (x,AMC^{\top}).
\]
Since
$\Gso\leq\Stab_{\SO_5(q)}(P)$ and
$\Gso^{(\infty)}$ is perfect, we have
$L:=\Gso^{(\infty)}\leq\Omega_4^+(q)$. By Lemma
\ref{lem_unique_nonabelian_factor}, $L$ has a unique non-abelian composition
factor $\PSL_2(q_0)$, where $q=q_0^m$.  By applying 
Goursat's lemma \cite[Proposition 1.6]{Serre} to the preimage of $L$ in $\SL_2(q)\times\SL_2(q)$, we deduce that up to conjugacy, one of the following cases occurs:
\begin{enumerate}
\item[(I)] 
 $L=\{I_2\circ A:A\in\SL_2(q_0)\}$ after interchanging the factors if necessary;
 \item[(II)]$L=\{A\circ\sigma(A):A\in\SL_2(q_0)\}
 \cong\PSL_2(q_0)$ for some $\sigma\in\Aut(\SL_2(q))$. 
\end{enumerate}

We first exclude Case \textup{(I)}. Take $\la(x,M)\ra\in\cB$. Put
\[
 u_1=\begin{pmatrix}1&1\\0&1\end{pmatrix},\qquad
 u_2=\begin{pmatrix}1&0\\1&1\end{pmatrix}.
\]
Both matrices are nonidentity elements of trace $2$ in
$\SL_2(q_0)$. If $x\ne0$, then $M$ is invertible, and each
$u_i$ maps $\la(x,M)\ra$ to a point
$\la(x,Mu_i^\top)\ra$ distinct from $\la(x,M)\ra$. If
$x=0$, then $M$
has rank one. Each of $u_1^\top$ and $u_2^\top$ has a unique invariant row line, and these two lines are distinct, so at most one of them fixes $\la(x,M)\ra$.
Choose $g\in\{u_1,u_2\}$ that does not fix $\la(x,M)\ra$. Then
$I_2\circ g\in L$ sends $\la(x,M)\ra$ to the distinct point
$\la(x,Mg^\top)\ra$. Since
$x^2=\det(M)$,
\[
 B\bigl((x,M),(x,Mg^\top)\bigr)
 =\bigl(4-\det(I_2+g^\top)\bigr)x^2
 =\bigl(2-\tr(g^\top)\bigr)x^2=0.
\]
Thus two distinct points of $\cB$ are collinear, a contradiction.
Hence Case \textup{(I)} is impossible.

We are therefore in Case \textup{(II)}. Write $q=p^e$ and $q_0=p^t$,
where $t\mid e$.  Here $L$ is a diagonal subfield
subgroup. For either projective factor, the standard subfield-normalizer
equality
\[
N_{\PGL_2(q)}(\PSL_2(q_0))=\PGL_2(q_0)
\]
shows that the normalizer of the corresponding projection of $L$ has
quotient of order at most two; see \cite{Dickson1904}*{\S255}. Moreover,
$C_{\Omega_4^+(q)}(L)\leq Z(\Omega_4^+(q))$; factor interchange has
order two, and $\Stab_{\SO_5(q)}(P)/\Omega_4^+(q)$ is a $2$-group.  Hence $N_{\Stab_{\SO_5(q)}(P)}(L)/L$ is a $2$-group. Since $L=\Gso^{(\infty)}$ is characteristic
in $\Gso$, we have $\Gso\leq N_{\Stab_{\SO_5(q)}(P)}(L)$, so every
abelian composition factor of $\Gso$ has order $2$. By Lemma \ref{lem_primitive} and the primitive-divisor argument at the beginning of this subsection, we deduce that $t=e$,
$q_0=q$, and
\[
 L=\{A\circ\sigma(A):A\in\SL_2(q)\}
 \cong\PSL_2(q)\cong\Omega_3(q).
\]

Every automorphism of $\SL_2(q)$ is a product of a field automorphism
and conjugation by an element of $\GL_2(q)$. Hence we may write
\[
 \sigma(A)= D^{-1}\psi(A)D,
\]
where $D\in\GL_2(q)$ and $\psi$ is the field automorphism obtained by
applying $a\mapsto a^{p^i}$ entrywise for some $0\leq i<e$.

Put $\eta:=\det(D)$ and choose $P_\infty\in\cB$. Since $L\unlhd G$
and $G$ is transitive on $\cB$, Lemma \ref{Nelson-lemma} gives
\[
 [L:\Stab_L(P_\infty)]\mid (q+1).
\]
By Dickson's classification \cite{Dickson1904}, $\Stab_L(P_\infty)$ is
either $L$ or a Borel subgroup isomorphic to $E_q:C_{(q-1)/2}$.
In either case, the stabilizer contains a Borel subgroup.
After conjugation in $L$, we may assume that the full preimage in
$\SL_2(q)$ of one such Borel subgroup consists of the matrices
\[
A_{\lambda,z}:=\begin{pmatrix}\lambda&0\\z&\lambda^{-1}\end{pmatrix},
 \qquad \lambda\in\F_q^*,\quad z\in\F_q.
\]
Write $P_\infty=\la(u,MD^{-\top})\ra$ and
$M=\begin{pmatrix}a&b\\c&d\end{pmatrix}$.\footnote{This choice of coordinates absorbs the conjugating matrix $D$ in the action, as in \eqref{eqn_Alaz}.} For every
$\lambda\in\F_q^*$ and $z\in\F_q$, there is a scalar
$\kappa_{\lambda,z}\in\F_q^*$ such that
\begin{equation}\label{eqn_Alaz}
 \bigl(u,A_{\lambda,z}M\psi(A_{\lambda,z}^{\top})D^{-\top} \bigr)
 =\kappa_{\lambda,z}(u,MD^{-\top}).
\end{equation}
 Observe that $A_{\lambda,z}M\psi(A_{\lambda,z}^{\top})$ is
\[
\begin{pmatrix}
 \lambda^{p^i+1}a&
 \lambda a z^{p^i}+\lambda^{1-p^i}b\\
 \lambda^{p^i}za+\lambda^{p^i-1}c&
 z^{p^i+1}a+\lambda^{-1}cz^{p^i}
 +\lambda^{-p^i}zb+\lambda^{-p^i-1}d
\end{pmatrix}.
\]

We claim that $u=0$. Suppose, to the contrary, that $u\ne0$. The first
coordinate then forces $\kappa_{\lambda,z}=1$ for every $\lambda,z$. By
setting $\lambda=1$ and varying $z\in \F_q$, we obtain 
\[
 a=0,\qquad cz^{p^i}+bz=0
 \quad\text{for every }z\in\F_q.
\]
If $i\ne0$, this identity forces $b=c=0$: otherwise
$z^{p^i-1}$ would be constant on $\F_q^*$, which is impossible for
$0<i<e$. Next, taking $z=0$ and varying $\lambda\in\F_q^\times$ in
\eqref{eqn_Alaz} gives $d=0$, since $0<p^i+1<q-1$. Thus $M=0$,
contradicting $Q(u,MD^{-\top})=0$. It remains to consider $i=0$. The
same identity gives $c=-b$, and the diagonal elements of the Borel
subgroup give $d=0$. If $b=0$, then again $M=0$, a contradiction. Hence
\[
 M=bJ,\qquad J=\begin{pmatrix}0&1\\-1&0\end{pmatrix},\qquad b\ne0.
\]
Since $AJA^\top=J$ for every $A\in\SL_2(q)$, the whole group $L$
fixes $P_\infty$. As $L\unlhd G$ and $G$ is transitive on $\cB$, it
follows that $L$ fixes every point of $\cB$. Since
$L\cong\PSL_2(q)$ is perfect, its action on any invariant
one-dimensional subspace is trivial. Hence every projective point
fixed by $L$ lies in $\PG(C_V(L))$. The fixed space of $L$ in this case is
\[
 C_V(L)=\{(x,cJD^{-\top}):x,c\in\F_q\}.
\]
The restriction of the quadratic form to this space is
$Q(x,cJD^{-\top})=x^2-\eta^{-1}c^2$. The projective line
$\PG(C_V(L))$ is therefore either external or secant, and contains at most two
singular points. It cannot contain all $q+1$ points of $\cB$. This contradiction proves that $u=0$. Hence
$P_\infty\in P^\perp$.
The hyperplane $P^\perp$ is $G$-invariant. Since $G$ is transitive on
$\cB$, we have $\cB\subseteq P^\perp$. Theorem \ref{thm_charact} now shows
that $\cB$ is either classical or a Kantor semifield BLT-set.
\end{proof}

\section{The \texorpdfstring{$A_5$}{A5}-case}\label{sec:4}
 Assume now that the $A_5$-condition introduced in Section 3 holds. The defining-characteristic case
was treated there.  The classification in \cites{LawPenttila2003,Betten2013,SmallBLT} establishes the conclusion in Theorem \ref{thm_main} when $q<71$.
Hence, in what follows, we will assume that
$p\ne5$ and $q\geq71$, and set
\[
 H:=\Gso^{(\infty)}\unlhd G.
\]
Let $\xi$ be as in Lemma \ref{lem_xisquare}; thus
$\big\la\xi \big\ra=\{y\in\F_{q^2}^{\times}:y^{q+1}=1\}$ is cyclic of order
$q+1$.  Let $\nu\in\blacksquare_q$ be fixed for the remainder of this section. By Theorem \ref{thm_charact}, we may assume that the vectors representing
$\cB$ span $V$; equivalently,
\begin{equation}\label{eqn_A5_full_span}
 \mathrm{Span}_{\F_q}\{\bfv:\la\bfv\ra\in\cB\}=V.
\end{equation}
 If $\PG(W)$ met $\mathcal B$ for a proper $G$-invariant subspace $W<V$, the transitivity  would imply  $\mathcal B\subseteq\PG(W)$, contradicting \eqref{eqn_A5_full_span}.
Thus
\begin{equation}\label{eqn_A5_avoids_invariant_subspaces}
 \cB\cap\PG(W)=\emptyset
 \qquad\text{for every proper $G$-invariant subspace $W<V$.}
\end{equation}
 Under the full-span hypothesis \eqref{eqn_A5_full_span}, it is enough to consider the geometric types (iv) and (vi) of Theorem \ref{thm_SO_5(q)}. More precisely, the geometric types (i)–(iii) and (v) force $q\in \{79, 239\}$, as shown in the following lemma. 
	\begin{lemma}\label{lem_A5_geometric_reduction} 
		Under the full-span hypothesis \eqref{eqn_A5_full_span}, every geometry of type \textup{(i)--(iii)} or \textup{(v)} in Theorem \ref{thm_SO_5(q)} forces $q\in \{79,239\}$. The remaining cases over these two fields are excluded by using Magma~\cite{Magma}.
	\end{lemma}
	
	\begin{proof}
		In types \textup{(i)} and \textup{(iii)}, let $R$ be the unipotent
		radical and $\rho$ the Levi projection.  Since $G$ preserves the relevant point or line, it normalizes $R$. Hence $\Gso\cap R\unlhd G$, and its orbits on $\mathcal B$ have a common $p$-power length dividing $q+1$. A unipotent element fixing
		$\la\bfv\ra$ fixes $\bfv$; hence \eqref{eqn_A5_full_span} gives
		$\Gso\cap R=1$, and $\rho$ is injective on $\Gso$.

	    By the stabilizer structures in Theorem
	    \ref{thm_SO_5(q)} and Lemma \ref{lem_unique_nonabelian_factor},
	    exactly one non-abelian composition factor occurs in these cases, and it is isomorphic to $A_5$. It follows that $H\cong A_5$ in types \textup{(i)},
		\textup{(ii)}, and \textup{(v)}, and $H\cong2.A_5$ in type \textup{(iii)}. By \cite{BHR407}*{Section 4, Table 4.4}, there are two $\Omega_3(q)$-conjugacy classes of subgroups isomorphic to $A_5$.  These two classes are fused in $\SO_3(q)$. Moreover, \[N_{\SO_3(q)}(A_5)=A_5.\]
       	Similarly, in type \textup{(iii)}, for $L\cong2.A_5$ acting faithfully on a two-dimensional module, we have
		\[
		 N_{\GL_2(q)}(L)=Z(\GL_2(q))L,
		 \qquad |N_{\GL_2(q)}(L)|=60(q-1).
		\]
		
		The preceding normalizer identities give
		$|\Gso|\mid60(q-1)$ in types \textup{(i)} and \textup{(iii)}, and
		$|\Gso|\mid120(q-1)$ in type \textup{(ii)}.

		For type \textup{(v)}, identify
		$\Omega_4^-(q)$ with $\PSL_2(q^2)$ and let
		$\widetilde H\cong2.A_5$ be the inverse image of $H$ in
		$\SL_2(q^2)$. Conjugation in $\SL_2(q^2)$ preserves trace and therefore cannot
		interchange the two conjugacy classes of elements of order five in
		$\widetilde H$. Thus every element of its normalizer induces an inner
		automorphism, and Schur's lemma gives
		$N_{\SL_2(q^2)}(\widetilde H)=\widetilde H$. Passing to the projective quotient gives
		$N_{\Omega_4^-(q)}(H)=H$. Since $H$ is characteristic in $\Gso$ and
		$|\mathrm O_4^-(q):\Omega_4^-(q)|=4$, we have $|\Gso|\mid240$. Hence
		\begin{equation}\label{eqn_A5_geometric_normalizer_bounds}
		 |\Gso|\mid120(q-1)
		 \quad\text{in types \textup{(i)}--\textup{(iii)}},
		 \qquad
		 |\Gso|\mid240
		 \quad\text{in type \textup{(v)}}.
		\end{equation}
		
		Suppose that $e>1$, and let $r$ be a primitive prime divisor of
		$p^{2e}-1$. By applying Lemmas \ref{lem_Zsigmondyprime} and \ref{lem_primitive},
		together with \eqref{eqn_A5_geometric_normalizer_bounds}, we obtain $r\mid240$; in types \textup{(i)--(iii)}, use
		$\operatorname{ord}_r(p)=2e$ to obtain $r\nmid q-1$. Since
		$2e\mid r-1$, we must have $(r,e)=(5,2)$. Now
		\eqref{eqn_divisiblecond} and
		$\gcd((p^2+1)/2,p^2-1)=1$ give $(p^2+1)/2\mid240$, whence $p=3$ and
		$q=9$, contradicting $q\geq71$. Thus $e=1$.
		
		Next, it follows from Lemma \ref{lem_primitive}, \eqref{eqn_orthogonal_core}, and
		$\gcd(q+1,q-1)=2$ that
		$(q+1)\mid 240$ and $\overline G=\overline{\Gso}$. Consequently, $ q\in\{79,239\}.$ A direct verification in Magma~\cite{Magma} excludes both primes in types \textup{(i)--(iii)} and \textup{(v)}. This completes the proof.
	\end{proof}

\subsection{The external-line case}\label{sec:external-A5-case}
Suppose that $V=U\perp W$, where $U=\F_{q^2}$ and $W=\F_q^3$. Equip $V$ with \[
 Q(u,\bfw)=u^{q+1}+Q_{A_5}(\bfw),
\]
where $Q_{A_5}$ is the invariant quadratic form in Subsection
\ref{sec:2_A5} (3). Choose $\delta\in\F_{q^2}$ with $\delta^2=\nu$.
In the basis $1,\delta$ of $U$, the norm form has Gram determinant
$-4\nu$, while the form on $W$ has Gram determinant $-4b_5$.
Hence, we deduce that 
\[
 \operatorname{disc}(Q)=16b_5\nu\in b_5\blacksquare_q,
 \qquad
 -\frac{\operatorname{disc}(Q)}2\blacksquare_q
 =-2b_5\square_q.
\]
The line $\PG(U)$ is external to $\cQ(4,q)$. By Theorem
\ref{thm_SO_5(q)}, we have 
\[
 K:=\Stab_{\SO_5(q)}(\PG(U))\cong(C_{q+1}\times\SO_3(q)).\la \tau\ra,
\]
where $\tau$ is an involution defined by $\tau(u,\bfw)=(u^q,-\bfw)$. The quotient of this stabilizer by the
$\Omega_3(q)$-factor is solvable. Since $H\cong A_5$ is perfect, it
follows that $H\leq\Omega_3(q)$. Thus $H$ fixes $U$ pointwise and acts
irreducibly on $W$. Moreover, $U=C_V(H)$; since $H\unlhd G$, both $U$
and $W=U^\perp$ are $G$-invariant.

For $y\in\big\la\xi \big\ra$, put $c_y(u,\bfw)=(yu,\bfw)$. Set
$T:=\la c_\xi\ra$, and choose $k\mid (q+1)$ so that $
 G\cap T=\la c_{\xi^k}\ra.$
 Since $G$ preserves $U$ and $W$, it normalizes  $T=\SO(U)\cong \SO_2^-(q)$, which acts trivially on $W$. Also $H\unlhd G$. Hence $J:=( G\cap T)H$ is normal in $G$ and
\[
 |\overline J|=\frac{60(q+1)}{k}.
\]
 Since $H=\Gso^{(\infty)}$ is characteristic in
$\Gso$, we have $\Gso\leq N_{\SO_5(q)}(H)$. The structure of $K$ and
$N_{\SO_3(q)}(H)=H$ give $
 N_{\SO_5(q)}(H)=(T\times H).\la\tau\ra.$ Because $H\leq G$, we have
\[
 \Gso\cap(T\times H)=( G\cap T)H=J,
\]
and therefore $[\Gso:J]\mid2$. The projective map is injective on
$\SO_5(q)$, so \eqref{eqn_orthogonal_core} now gives
\[
 [\overline G:\overline J]
 = [\overline G:\overline{\Gso}][\Gso:J]\mid2e.
\]
The transitive action of $\overline{G}$ on $\cB$ implies that $(q+1)\mid |\overline G|$; together with the order of
$\overline J$, it further gives
\begin{equation}\label{eqn_caseI_k_divides}
 k\mid60[\overline G:\overline J]\mid120e.
\end{equation}

Suppose that $P=\la\bfv\ra=\la(u,\bfw)\ra\in\cB$. By \eqref{eqn_A5_avoids_invariant_subspaces}, both $u$ and
$\bfw$ are nonzero, and $Q_{A_5}(\bfw)=-u^{q+1}\ne0$. Hence $\Stab_H(\bfw)$ embeds in the cyclic
group $\SO(\bfw^\perp)$, and
\begin{equation}\label{eqn_StabH}
 |\Stab_H(\bfw)|\in\{1,2,3,5\}.
\end{equation}
If nontrivial, this stabilizer has fixed space $\F_q\bfw$ on $W$.
Since $H$ fixes $U$ pointwise,
\begin{equation}\label{eqn_external_stabilizers}
 \Stab_H(P)=\Stab_H(\bfv)=\Stab_H(\bfw).
\end{equation}

\begin{lemma}\label{lem_A5}
Let $P=\la\bfv\ra\in\cB$, with $\bfv=(u,\bfw)$ as above. Under the full-span hypothesis \eqref{eqn_A5_full_span}, if $(q+1)/k>2$, then
$60\mid k$, $u^{q+1}\in-2b_5\blacksquare_q$, and
$\Stab_J(P)=1$. Moreover, for every $g\in J\setminus\{1\}$,
$B(\bfv,g\bfv)\in-2b_5\square_q$.
\end{lemma}

\begin{proof}
 Since $H$ is perfect and $\Stab_H(P)=\Stab_H(\bfv)$ by \eqref{eqn_external_stabilizers}, Theorem \ref{thm_square_class} gives that 
\begin{equation}\label{eqn_external_H_class}
 B(\bfv,h\bfv)\in-2b_5\square_q
 \qquad\text{whenever }h\in H\text{ and }hP\ne P
\end{equation}
in the present model. 

Since both components $u,\bfw$ of $\bfv$ are nonzero, $c_yP=P$ holds only for
$y=1$. Thus $P,c_{\xi^k}P,c_{\xi^{2k}}P$ are distinct, and
\begin{align*}
 B(\bfv,c_{\xi^k}\bfv)
 &=B(c_{\xi^k}\bfv,c_{\xi^{2k}}\bfv)
   =(\xi^k+\xi^{-k}-2)u^{q+1},\\
 B(\bfv,c_{\xi^{2k}}\bfv)
 &=(\xi^{2k}+\xi^{-2k}-2)u^{q+1}.
\end{align*}
Since $(q+1)/k>2$, we have $2k\leq q$. By Lemma
\ref{lem_xisquare}, applied with exponent $2k$, $\xi^{2k}+\xi^{-2k}-2$ is a nonsquare.
By applying Lemma \ref{lem_disc} to these three distinct points, we deduce 
\begin{equation}\label{eqn_caseI_u_class}
 u^{q+1}\in-2b_5\blacksquare_q.
\end{equation}
No $h\in H$ satisfies $h\bfw=-\bfw$. Indeed, if such an $h$ exists, then $hP\ne P$ and
\[
B(\bfv,h\bfv)=4u^{q+1}\in-2b_5\blacksquare_q,
\]
contradicting \eqref{eqn_external_H_class}.

Assume that $k$ is odd. Then $c_yhP=P$ forces
$y^2=1$, and the case $y=-1$ has just been excluded. Thus
$\Stab_J(P)=\Stab_H(\bfw)$. The normality of $\overline J$ in $\overline G$ implies that each $\overline J$-orbit on $\cB$ has
length
\[
 \frac{60(q+1)}{k|\Stab_H(\bfw)|},
\]
dividing $q+1$. Hence $60/|\Stab_H(\bfw)|$ divides $k$, which is impossible for odd
$k$ because $|\Stab_H(\bfw)|\in\{1,2,3,5\}$ in \eqref{eqn_StabH}. Therefore $k$ is even.

For every $y\in\la\xi^k\ra\setminus\{1\}$, we obtain from Lemma \ref{lem_xisquare} and
\eqref{eqn_caseI_u_class} that
\[
 B(\bfv,c_y\bfv)=(y+y^{-1}-2)u^{q+1}
 \in-2b_5\square_q.
\]
The same square class holds for $c_yh$ whenever $c_yhP\ne P$. If
$hP=P$, then $h\bfw=\bfw$ and
$B(\bfv,c_yh\bfv)=B(\bfv,c_y\bfv)$; if $y=1$, it follows from
\eqref{eqn_external_H_class}. Otherwise the points
$P,c_yP,c_yhP$ are distinct, so Lemma \ref{lem_disc} applies. Hence
\begin{align}
B(\bfv,c_yh\bfv)
&=(y+y^{-1})u^{q+1}+B_{A_5}(\bfw,h\bfw)\notag\\
&=(y+2+y^{-1})u^{q+1}+Q_{A_5}(\bfw+h\bfw)
  \in-2b_5\square_q.\label{eqn_cond_A5}
\end{align}

Suppose that $\Stab_H(\bfw)$ is nontrivial. For stabilizers of orders
$2,3,$ and $5$, the corresponding normalizers in $A_5$ are isomorphic
to $V_4,S_3,$ and $D_{10}$, respectively. Hence some involution
$h\notin\Stab_H(\bfw)$ normalizes $\Stab_H(\bfw)$. It preserves
$C_W(\Stab_H(\bfw))=\F_q\bfw$, and therefore $h\bfw=-\bfw$. However, we have already shown that no such $h$ exists. Thus $\Stab_H(\bfw)=1$.

If $c_yhP=P$, then $y=\pm1$. Since there is no $h\in H$ with $h\bfw=-\bfw$, we have $y=1$. Thus
$\Stab_J(P)=1$, and every nonidentity element of $J$ has the asserted square class. Finally,
\[
 \frac{60(q+1)}{k}=|P^{\overline J}|\mid (q+1),
\]
so $60\mid k$.
\end{proof}

\begin{lemma}\label{lem_caseI_small_torus}
Assume that $q\geq71$ and $(q+1)/k\leq2$. Then
$(q,k)=(239,120)$.
\end{lemma}

\begin{proof}
First, \eqref{eqn_caseI_k_divides} and $(q+1)/k\leq2$ imply $q+1\leq240e$. For $P=\la(u,\bfw)\ra\in\cB$,
we deduce from \eqref{eqn_external_stabilizers} that $|P^H|=|\bfw^H|$. By the normality of $\overline H$ in $\overline G$ and \eqref{eqn_StabH}, we have
\[
 |\bfw^H|\in\{60,30,20,12\},\qquad |\bfw^H|\mid (q+1).
\]
If the orbit length is not $30$, then $4\mid (q+1)$; if it is $30$, then
$3\mid (q+1)$. In either case $e$ is odd.

Suppose that $e>1$. The inequality $p^e+1\leq240e$ excludes $e\geq7$.
For $e=3$ or $5$, using $q\geq71$ and $p\ne5$ leaves only
$(p,e)=(7,3)$ and $(3,5)$. But neither $344$ nor $244$ is divisible by
any of $12,20,30,60$. Thus $e=1$.

With $e=1$, we have $k\mid120$. Together with $q\geq71$,
$q\equiv\pm1\pmod{10}$, $(q+1)/k\leq2$, and the possible orbit lengths
above, this leaves
\[
 (q,k)=(79,40)\quad\text{or}\quad(239,120).
\]
If $(q,k)=(79,40)$, then $|\bfw^H|=20$ and
$|\Stab_H(\bfw)|=3$. The normalizer of this subgroup in $H$ contains an
involution $h\notin\Stab_H(\bfw)$, for which $h\bfw=-\bfw$. Since
$c_{-1}h\bfv=-\bfv$, the subgroup
$\la\Stab_H(\bfw),c_{-1}h\ra$ of order six fixes $P$. Therefore
$|P^{\overline J}|\leq20$, and
\[
 |\cB|=|P^{\overline G}|
 \leq[\overline G:\overline J]|P^{\overline J}|
 \leq40,
\]
a contradiction. Hence only $(q,k)=(239,120)$ remains.
\end{proof}

Assume $(q+1)/k>2$. Lemma \ref{lem_A5} shows that
$-u^{q+1}/(2b_5)$ and $\nu$ are nonsquares. After rescaling
$\bfv$, assume that
$-u^{q+1}/(2b_5)=\nu$. Since $W$ is irreducible and the involutions generate $H$, there is an involution $h\in H$ such that $h\bfw\notin\F_q\bfw$. Set
$a=-Q_{A_5}(\bfw+h\bfw)/(2b_5)$. By
Lemma \ref{lem_A5}, we have $c_yhP\ne P$ for every
$y\in\la\xi^k\ra$; hence we obtain from \eqref{eqn_cond_A5} that
\begin{equation}\label{eqn_caseI_square_condition}
\nu(y+2+y^{-1})+a\in\square_q
\qquad\text{for every }y\in\la\xi^k\ra.
\end{equation}
The $(+1)$-eigenspace of $h$ is a non-degenerate $1$-space over $\F_q$. The projection of $\bfw$ onto it is nonzero;  otherwise $h\bfw=-\bfw$, which is a contradiction to Lemma \ref{lem_A5}. Thus $a\ne0$. Taking $y=1$ above also gives
$a\ne-4\nu$.

Define the auxiliary set
\begin{equation}\label{eqn_caseI_admissible_scalars}
\mathcal A_{q,k}:=
\{a\in\F_q:a+\nu(y+2+y^{-1})\in\square_q
\text{ for every }y\in\la\xi^k\ra\}\setminus\{0,-4\nu\}.
\end{equation}
Thus $\mathcal A_{q,k}\ne\varnothing$ is necessary in the external-line
case with $(q+1)/k>2$.

\begin{corollary}\label{cor_caseI_prime_branches}
In the external-line case, $q$ is prime and $q<(2k-1)^2$. If
$(q+1)/k>2$, then $q\equiv-1\pmod{60}$; moreover, if $m$ is the number
of $\overline J$-orbits on $\cB$, then
$m\in\{1,2\}$, $k=60m$, and $q<(120m-1)^2$.
\end{corollary}

\begin{proof}
If $(q+1)/k\leq2$, Lemma \ref{lem_caseI_small_torus} gives
$(q,k)=(239,120)$, and the first assertion follows. Suppose that
$(q+1)/k>2$. Apply Lemma \ref{lem_caseI_torus_sum} with
$\alpha=\nu$ and $\beta=a+2\nu$. Since
$\beta^2-4\alpha\alpha^q=a(a+4\nu)\ne0$, and every term in
\eqref{eqn_caseI_square_condition} is a nonzero square, we obtain
\[
 \frac{q+1}{k}\leq\frac{(2k-1)\sqrt q+1}{k}.
\]
Thus $q\leq(2k-1)^2$.

By Lemma \ref{lem_A5}, we have $60\mid k$, and then
$q\equiv-1\pmod{60}$. Hence $e$ must be odd. The unit group modulo $60$ has exponent $4$, so $p^4\equiv1\pmod{60}$. Since $p^e\equiv-1\pmod{60}$ and $e$ is odd, we have $e\equiv 1$ or $3\pmod{4}$. In both cases, $p\equiv -1\pmod{60}$. Thus $p\geq59$. Combining $k\leq120e$ with the
preceding bound gives
\[
 p^e\leq(240e-1)^2.
\]
This is impossible for $e\geq5$. If $e=3$, the same inequality and
$p\equiv-1\pmod{60}$ force $p=59$. But then
\[
 k\leq\gcd(59^3+1,360)=180,
\]
whereas $59^3>(2\cdot180-1)^2$, again a contradiction. Hence $e=1$
and $q$ is prime. The equality $q=(2k-1)^2$ is now impossible, so
$q<(2k-1)^2$.

Let $m$ be the number of $\overline J$-orbits on $\cB$. Each orbit has
length $60(q+1)/k$, and hence $k=60m$. Since
$\overline J\unlhd\overline G$, the quotient
$\overline G/\overline J$ is transitive on these $m$ orbits. Therefore
$m\mid[\overline G:\overline J]\mid2$. Substituting $k=60m$ gives the
remaining assertions.
\end{proof}

\subsection{The hyperbolic case}\label{sect:hyperbolic}
Let $M_2(q)$ denote the set of all $2\times 2$ matrices over $\F_q$. Let $V=\F_q\oplus\mathcal M$, where 
\[
\mathcal M= M_2(q) \text{ or }\left\{
\begin{pmatrix}
\alpha&\beta\\
-\beta^q&\alpha^q
\end{pmatrix} :\alpha,\beta\in \F_{q^2}
\right\}.\]   Then $\dim_{\F_q}\mathcal M=4$, and the quadratic form
\[
 Q(x,M)=x^2-\det(M)
\]
defines $\cQ(4,q)$ with $\operatorname{disc}(Q)\in2\square_q$. In this
model, $\Omega_4^+(q)\cong \SL_2(q)\circ \SL_2(q) $ (or $\SU_2(q)\circ\SU_2(q)$)\footnote{The choice of $\mathcal M$ depends on the explicit representation of $2.A_5$ used below; the relevant representations are given in Subsection \ref{sec:2_A5}.} acts by
\[
 (g_1\circ g_2)(x,M)=(x,g_1Mg_2^{-1}).
\]
Here $g_1,g_2\in\SL_2(q)$ in the first model, and
$g_1,g_2\in\SU_2(q)$ in the second.
See \cite{KL129}*{Section~2.10}. Both summands are $G$-invariant. By \eqref{eqn_A5_avoids_invariant_subspaces}, every point of $\cB$
has a unique representative
\begin{equation}\label{eqn_hyperbolic_normalized_point}
 \bfv_M=(1,M),\qquad M\in \mathcal{M}.
\end{equation}

Since $H$ is perfect, $H\leq\Omega_4^+(q)$. By Lemma \ref{lem_subPSL} and Lemma
\ref{lem_unique_nonabelian_factor}, every nontrivial projection of $H$ to a projective factor is $A_5$. By Goursat's lemma \cite[Proposition 1.6]{Serre}, there are three possibilities: a graph subgroup, the central product of the two factors isomorphic to $2.A_5$, or a subgroup of one factor.

 For the graph-subgroup case, use the model $\mathcal{M}=M_2(q)$. Write $H=\{ a\circ\sigma(a):\, a\in L_1\}$, where $L_1,L_2 \leq\SL_2(q)$ are isomorphic to $2.A_5$ and $\sigma:\,L_1\rightarrow L_2$ is an isomorphism. Suppose that $\sigma(a)=M^{-1}aM$ for some $M\in \GL_2(q)$ and each $a\in L_1$. Then $C_V(H)=\F_q\oplus \F_q M$; see the proof
of Proposition \ref{prop_case6}. If $\det(M)\in\square_q$, then $C_V(H)$ is hyperbolic, so $G$ stabilizes a pair of non-collinear singular points. Lemma \ref{lem_A5_geometric_reduction}, including its exclusion of the remaining primes $79$ and $239$, rules out this case.
If $\det(M)\in\blacksquare_q$, then $C_V(H)$ is anisotropic, and $\PG(C_V(H))$ is a $G$-invariant external line.
This reduces to the external-line case as discussed in Subsection \ref{sec:external-A5-case}.
Thus, apart from the external-line case, the following three possibilities remain, up to conjugacy and factor interchange: 
\begin{enumerate}
\item[(I)]  $H=\{a\circ\sigma(a):a\in L_1\}\cong A_5$, where $L_1,L_2\cong2.A_5$ and the isomorphism $\sigma: L_1\rightarrow L_2$  is not induced by conjugation by any matrix in $\GL_2(q)$;
\item[(II)] $H=L_1\circ L_2$ with
$L_1,L_2\cong2.A_5$;
\item[(III)] $H=I_2\circ L$ with
$L\cong2.A_5$.
\end{enumerate}

The following lemma applies to Cases \textup{(II)} and \textup{(III)}.

\begin{lemma}\label{lem_hyperbolic_factor_orbit}
Suppose that a subgroup $H_0=L\circ I_2$ or $I_2\circ L$, with
$L\cong2.A_5$, is contained in $H$. Then $H_0$ acts semiregularly on $\cB$, and each $H_0$-orbit is a partial BLT-set of size $120$. In particular, $q\equiv-1\pmod{120}$. 
With the representatives in \eqref{eqn_hyperbolic_normalized_point},
the value $B(\bfv_M,\bfv_N)$ is a nonzero square whenever $\la\bfv_M\ra$ and $\la\bfv_N\ra$ are distinct points in  the same $H_0$-orbit.
\end{lemma}

\begin{proof}
By symmetry, assume $H_0:=L\circ I_2$. If $g\circ I_2$ fixes $P=\la(1,M)\ra$, then $gM=M$, and hence $g=I_2$ since $\det(M)=1$. It follows that $H_0$ acts semiregularly on $\cB$, so its orbits have length $120$ and
$120\mid (q+1)$.

 Since $H_0$ is perfect, Theorem \ref{thm_square_class} gives 
 \[
  B(\bfv_M,g\bfv_M)\in
  -\frac{\operatorname{disc}(Q)}2\blacksquare_q=\square_q,\quad\, \bfv_M=(1, M)
 \]
 for every nonidentity $g\circ I_2\in H_0$, because
 $\operatorname{disc}(Q)\in2\square_q$ and $q\equiv3\pmod4$.
\end{proof}
For Cases \textup{(II)} and \textup{(III)},  since $q\equiv-1\pmod{120}$, $e$ is odd by considering the reduction modulo $8$. Thus $e\equiv 1$ or $3\pmod 4$.  Since every
unit modulo $120$ has fourth power $1$, the congruence
$p^e\equiv-1\pmod{120}$ implies in either case that
$p\equiv-1\pmod{120}$. Consequently,
\begin{equation}\label{eqn_hyperbolic_prime_congruence}
e\, \text{ is odd},\qquad p\equiv-1\pmod{120},\qquad p\geq239.
\end{equation}

For $L\cong2.A_5\leq\SU_2(q)\cong \SL_2(q)$, we have 
\begin{equation}\label{eqn_2A5_self_normalizer}
 N_{\SU_2(q)}(L)=L.
\end{equation}
Indeed, the outer automorphism interchanges the two classes of elements of order five, whose distinct traces are $b_5$ and $-1-b_5$. Since conjugation preserves trace, a normalizing element induces an inner automorphism. After multiplication by an element of $L$, the normalizing element centralizes $L$ and is therefore scalar by Schur's lemma; the scalar subgroup $\{\pm I_2\}$ lies in $L$.

\begin{lemma}\label{lem_A5_twisted_diagonal_exclusion}
Case \textup{(I)} does not occur.
\end{lemma}
\begin{proof}
 In Case \textup{(I)}, set $L=L_1$ and choose $C\in\GL_2(q)$ such that $C^{-1}L_1C=L_2$. Then $\alpha(a)=C\sigma(a)C^{-1}$ defines an automorphism of $L$ representing the nontrivial element of $\operatorname{Out}(L)$;
otherwise $\sigma$ would be induced by matrix conjugation. Let $U$ be the natural two-dimensional $\F_qL$-module. The two factors of the graph action afford $U$ and its twist $U^\alpha$.

Using the ordinary character table in \cite{ATLAS}*{p.~2} and,
when $p=3$, the $3$-Brauer table in
\cite{BrauerAtlas}*{p.~2}, we can identify
$\mathcal M$ with the four-dimensional fully deleted permutation
module. Since $p\ne5$, we can identify this module with
\[
 \mathcal{M}=\{(y_1,\ldots,y_5)\in\F_q^5:\textstyle\sum_i y_i=0\},
 \qquad \F_q^5=\F_q\mathbf1\oplus \mathcal{M};
\]
see \cite{KL129}*{Section~5.3, pp.~185--187}.  By Schur's lemma, the centralizer of $H$ on $\mathcal{M}$ consists of scalars.   Since $H$ is perfect and $\mathrm{Out}(A_5)=2$, we deduce from  \cite{BHR407}*{Table 8.22} that  \[ [N_{\PGO_5(q)}(\overline H): \overline{H} ] \mid 4.\]
Together with the Frobenius automorphisms, we have $|\overline{G}| \mid 240e$.  If $e>1$, let $r$ be a primitive prime divisor of $p^{2e}-1$. Then
$r\mid (q+1)$, $2e\mid r-1$, and $r\nmid e$. Thus $r\mid240$, which
forces $r=5$ and $e=2$. It follows that $p^2+1\mid480$, forcing
$p=3$ and $q=9$, contrary to $q\geq71$. Hence $e=1$, and
$(q+1)\mid 240$ leaves $q\in\{79,239\}$. In either case
$\operatorname{lcm}(60,q+1)=240$; since
$\overline H\leq\overline G\leq N_{\PGO_5(q)}(\overline H)$, transitivity now gives
$\overline G=N_{\PGO_5(q)}(\overline H)$.

By \eqref{eqn_hyperbolic_normalized_point}, every point of $\cB$ has a
unique representative $(1,\mathbf y)$, with $\mathbf y\in \mathcal{M}$. The two
relative square classes of the invariant form are represented by
\[
 Q_d(x,\mathbf y)=x^2+d\sum_{i=1}^5y_i^2,
 \qquad d\in\{1,\nu\}.
\]
For $q\in\{79,239\}$, an exhaustive computation in
Magma~\cite{Magma} checks both classes and finds no orbit of length $q+1$
satisfying the partial BLT condition. Hence Case \textup{(I)} is impossible.
\end{proof}

\begin{lemma}\label{lem_A5_product_exclusion}
Case \textup{(II)} does not occur.
\end{lemma}

\begin{proof}
 In this case, we have $H=L_1\circ L_2$ and
$|\overline H|=7200$, where $L_i\cong2.A_5$. By Lemma \ref{lem_hyperbolic_factor_orbit}, $q\equiv-1\pmod{120}$. Considering the projection onto each factor and applying \eqref{eqn_2A5_self_normalizer}, we obtain  $N_{\Omega_4^+(q)}(H)=H$. Using
$|\mathrm O_4^+(q):\Omega_4^+(q)|=4$ together with \eqref{eqn_orthogonal_core}, we deduce that
\[
 [\overline G:\overline H]\mid4e,\qquad (q+1)\mid 28800e.
\]

Since \eqref{eqn_hyperbolic_prime_congruence} holds, we have $p\geq239$ and $e$ is odd. If $e>1$, then $e\geq3$ and $p^e+1>28800e$, which is a contradiction to the preceding divisibility. Thus $e=1$, and the divisibility above leaves the following primes:
\[
 q\in\{239,359,479,599,719,1439,2399,2879,4799\}.
\]

For $P=\la\bfv_M\ra$ as in \eqref{eqn_hyperbolic_normalized_point}, the nonzero first coordinate implies that
\[
 \Stab_{\overline H}(P)
 \cong(M^{-1}L_1M\cap L_2)/\{\pm I_2\}.
\]
Since $[\overline G:\overline H]\mid4$, the number $m$ of
$\overline H$-orbits on $\cB$ lies in $\{1,2,4\}$. If $r$ is their common
point-stabilizer order, then
$r\in\{1,2,3,4,5,6,10,12,60\}$ and $q+1=7200m/r$. This excludes
$q=359$. For $q=239$ and $479$, it forces $r=60$, and hence
$M^{-1}L_1M=L_2$ for every point matrix $M$.
For any two such matrices $M,N$ representing points $\la\bfv_M\ra,\la\bfv_N\ra\in\cB$, the element $M^{-1}N$ normalizes $L_2$
and hence lies in $L_2$ by \eqref{eqn_2A5_self_normalizer}. Thus $\cB$ is
contained in a single $120$-point orbit,  contrary to $q+1>120$. This leaves $q\in\{599,719,1439,2399,2879,4799\}$.

For each remaining prime and each embedding class of $L_2$, an exact computation in Magma~\cite{Magma} shows that 
\[
 \{X\in\SL_2(q):\tr(X)=1,\ 
   \chi(2-\tr(Xg))\in\{0,1\}\ \text{for all }g\in L_2\}
 \subseteq\{X\in L_2:\tr(X)=1\},
\]
and the set on the right consists of the $20$ elements of order $6$ in $L_2$.\footnote{In $\SL_2(q)$, since $-3\in \blacksquare_q $ by $q\equiv -1 \pmod{3}$, each element of $\SL_2(q)$ with trace $1$ can be written as 
	\[ \begin{pmatrix}
		a & b \\ b^{-1}(a-a^2-1) & 1-a
	\end{pmatrix},\, a\in\F_q,\, b\in \F_q^{\times}.\]}

Let $P=\la(1,M)\ra\in\cB$, let $s\in L_1$ have trace one, and let
$g\in L_2$. Then
\[
 B((1,M),(1,sMg))=2-\tr(M^{-1}sMg).
\]
By Theorem \ref{thm_square_class}, this is zero if $s\circ g^{-1}$ fixes
$P$ and a nonzero square otherwise. Hence the above inclusion implies $M^{-1}sM\in L_2$. Conjugation
by $M$ maps the $20$ trace-one elements of $L_1$ onto those of $L_2$. Since these elements generate $2.A_5$, we deduce that  $M^{-1}L_1M=L_2$. As a consequence, for any two point matrices $M,N$, we have
$M^{-1}N\in N_{\SL_2(q)}(L_2)=L_2$. Thus all points lie in one $120$-point orbit, contrary to $q+1>120$. So Case \textup{(II)} is impossible.
\end{proof}

In Case \textup{(III)},
$H=I_2\circ L$ and $L\cong2.A_5\leq\SU_2(q)$. By Lemma \ref{lem_hyperbolic_factor_orbit}, we have $q\equiv-1\pmod{120}$.  Write
\[
 S=\{s\in\SU_2(q):\pi(s\circ I_2)\in\overline G\},
 \qquad Z=\{\pm I_2\}\leq S.
\]
Here $Z\leq S$ because $(-I_2)\circ I_2=I_2\circ(-I_2)\in H$. A direct
calculation gives
\[
 C_{\PGO_5(q)}(\overline H)
 =\overline{\SU_2(q)\circ I_2}.
\]
The trace argument following \eqref{eqn_2A5_self_normalizer} also
excludes an outer automorphism here: the two conjugacy classes of
elements of order five in $L$ have distinct traces $2b_5$ and
$-2(1+b_5)$ on the four-dimensional space $\mathcal M$, and an outer
automorphism interchanges these classes. Hence
\begin{equation*}
 N_{\PGO_5(q)}(\overline H)
 =C_{\PGO_5(q)}(\overline H)\,\overline H.
\end{equation*}
It follows that $\overline{G}\cap \PGO_5(q)=\overline{ S\circ L}$. Since $S\circ L$ fixes the $1$-space $\F_q\oplus\{0\}$ pointwise, it can be identified with its projective image $\overline{S\circ L}$. Hence
\begin{equation}\label{eqn_one_factor_normalizer}
 \overline{S\circ L}\unlhd\overline G,
 \qquad [\overline G:\overline{S\circ L}]\mid e.
\end{equation}
Moreover,
\begin{equation}\label{eqn_one_factor_centralizer}
 \overline{S\circ I_2}
 =\overline G\cap C_{\PGO_5(q)}(\overline H)
 \unlhd\overline G.
\end{equation}
It is obvious that $S$ is solvable since $\Gso^{(\infty)}=H=I_2\circ L$, where $L\cong 2.A_5$. 

For each $y\in\big\la\xi \big\ra$, define
\[
 d_y=\begin{pmatrix}y&0\\0&y^{-1}\end{pmatrix}.
\]

\begin{lemma}\label{lem_hyperbolic_complement_reduction}\label{lem_hyperbolic_toral_core}
	In Case \textup{(III)}, up to conjugacy in the first factor, we may
take the nonsplit maximal torus $T=\la d_\xi\ra\leq\SU_2(q)$ of order
$q+1$ so that
\[
 S\leq N_{\SU_2(q)}(T),\qquad
 Z\leq D:=S\cap T,\qquad [S:D]\leq2.
\]
The subgroup $D$ is cyclic, and $D=S$ if $S$ is cyclic. Moreover,
\[
 \overline{D\circ L}\unlhd\overline G,\qquad
 [\overline G:\overline{D\circ L}]\mid2e.
\]
\end{lemma}

\begin{proof}
Fix $P=\la\bfv_M\ra\in\cB$. Since the first coordinate of $\bfv_M$ equals $1$, we have $\Stab_{S\circ L}(\bfv_M)=\Stab_{S\circ L}(P)$.  By Theorem 
\ref{thm_square_class},  there is a homomorphism
\[
\epsilon:\, S\circ L\longrightarrow\{\pm1\}.
\]
Recall from Lemma \ref{lem_hyperbolic_factor_orbit} that
$q\equiv -1 \pmod{120}$. Since $\operatorname{disc}(Q)\in2\square_q$, we have
$\epsilon(a)=\chi(B(\bfv_M,a\bfv_M))$ whenever $aP\ne P$.
Moreover, $\epsilon$ is trivial on the perfect subgroup $H=I_2\circ L$.

We first exclude elements of order $p$ and non-central elements whose
order divides $q-1$.   If $p\mid|S|$, then $S$ must contain an element $g$ of order $p$, necessarily unipotent with trace $2$. Since $gM \ne M$, the points $P$ and $(g\circ I_2)P$ are distinct, but
\[
 B(\bfv_M,(g \circ I_2)\bfv_M)=2-\tr(g)=0,
\]
contradicting the BLT condition. Thus $p\nmid|S|$.
Next suppose that $g\in S\setminus Z$ has order dividing $q-1$, and put $t=\tr(g)=\lambda+\lambda^{-1}$, where $\lambda,\lambda^{-1}$ are two eigenvalues of $g$ in $\F_q^\times$.  Since $g\notin Z$, the points
\[ P,\quad (g\circ I_2)P, \quad (g\circ (-I_2))P\]
are distinct. By Lemma \ref{lem_disc}, we have  $\epsilon(g\circ I_2)=\epsilon(g\circ -I_2)$, which implies $\chi(2-t)=\chi(2+t)$. However, $t^2-4=(\lambda-\lambda^{-1})^2$ is a nonzero square, whereas
$-1$ is a nonsquare. Hence
\[
 \chi\bigl((2-t)(2+t)\bigr)=\chi(4-t^2)=-1,
\]
which is a contradiction.

By Dickson's classification \cite{Dickson1904}*{Chapter~XII}, a solvable subgroup $S/Z$ of $\PSL_2(q)$ with these properties
is contained in the normalizer of a nonsplit torus or is isomorphic to $A_4$ or $S_4$. It remains to exclude the
last two possibilities, whose preimages are $S\cong2.A_4$ and
$S\cong2.S_4$. In either case, $|S\circ L|=60|S|$, so it follows  from the transitivity of $\overline{G}$ and
\eqref{eqn_one_factor_normalizer} that  $
 (q+1)\mid 60e|S|.$
By \eqref{eqn_hyperbolic_prime_congruence}, $p\geq239$ and $e$ is odd.
If $e>1$, then $e\geq3$ and
$p^e+1>2880e\geq60e|S|$, which is  a contradiction to $(q+1) \mid 60e|S|$. Thus $e=1$, and by 
\eqref{eqn_one_factor_normalizer} we have 
$\overline G=\overline{S\circ L}$. The divisibility above, together with
$q\equiv-1\pmod{120}$, leaves the following primes:
\[
\begin{array}{c|l}
S&q\\ \hline
2.A_4&239,359,479,719,1439\\
2.S_4&239,359,479,719,1439,2879.
\end{array}
\]

We first claim that  the homomorphism $\epsilon$ is  trivial on $S\circ I_2$ in both exceptional cases. For $S\cong2.A_4 \cong \SL_2(3)$, this follows from $S/[S,S]\cong C_3$. For
$S\cong2.S_4$, we have $S/[S,S]\cong C_2$. Choose $s\in S$  corresponding to a transposition in $S/Z\cong S_4$. Then $s\in \SL_2(q)$ has order $4$ and has trace zero, so $\epsilon(s\circ I_2)=\chi(2)=1$, since $q\equiv-1\pmod8$.
Thus $\epsilon$ is trivial on all of $S\circ L$. Consequently, by $\chi(-\mathrm{disc}(Q)/2)=-1$, we have 
\begin{equation}\label{eqn_chiA4S4}
 \chi\bigl(2-\tr(M^{-1}sMg)\bigr)\in\{0,1\}
 \qquad(s\in S,\ g\in L),
\end{equation}
because the expression inside $\chi$ equals $B(\bfv_M,(s\circ g^{-1})\bfv_M)$; it is zero when this element fixes $P$ and a nonzero square otherwise.

Now use the model $V=\F_q\oplus M_2(q)$, identifying the two
$\SU_2(q)$-factors with $\SL_2(q)$ via
$\SU_2(q)\cong\SL_2(q)$  for simplicity. For $t\in\F_q$, we define 
\[
 \mathcal E_L(t):=\{X\in\SL_2(q):\tr(X)=t,\quad
   \chi(2-\tr(Xg))\in\{0,1\}\text{ for every }g\in L\}.
\]
The condition \eqref{eqn_chiA4S4} implies  that
$M^{-1}sM\in\mathcal E_L(\tr(s))$ for every $s\in S$.
The two $\SL_2(q)$-conjugacy classes of $L$ are fused in $\GL_2(q)$, cf. \cite{Dickson1904}*{Chapter~XII}. Since simultaneous conjugation of $X$ and $L$ preserves both $\tr(X)$ and  every trace $\tr(Xg)$ with $g\in L$, it suffices to use
the representative $L=\la A_{\mathrm s},B_{\mathrm s}\ra$ from
Subsection \ref{sec:2_A5}. Choose $\tau\in\F_q$ with $\tau^2=2$.
Exhaustive computation gives
\[
\begin{aligned}
 \mathcal E_L(1)&=\{g\in L:\tr(g)=1\}
 &&\text{for }q\in\{239,359,479,719,1439\},\\
 \mathcal E_L(\tau)&=\varnothing
 &&\text{for }q\in\{239,359,479,719,1439,2879\}.
\end{aligned}
\]
These computations test all matrices of the indicated traces against
all $120$ elements of $L$.\footnote{{For $t=1$ and $t=\tau$, the
discriminant $t^2-4$ is $-3$ and $-2$, respectively, and both values are
nonsquares for the listed primes.}
Thus every matrix of trace $t$ in $\SL_2(q)$ has the unique form
$\left(\begin{smallmatrix}a&b\\
b^{-1}(a(t-a)-1)&t-a\end{smallmatrix}\right)$ with
$a\in\F_q$ and $b\in\F_q^{\times}$, giving $q(q-1)$ matrices to test.
The choice between $\tau$ and $-\tau$ does not affect emptiness, since
$-I_2\in L$.}

If $S\cong2.A_4$, the trace-one elements of $S$ generate $S$. The first computational equality therefore implies that $M^{-1}SM\leq L$. Hence the subgroup
\[
 \{s\circ(M^{-1}sM):s\in S\}
\]
has order $|S|/2=12$ and fixes $P$. Since $\overline G=\overline{S\circ L}$, the orbit-stabilizer
theorem gives
\[
 |P^{\overline G}|\leq\frac{24\cdot120/2}{12}=120<q+1,
\]
which contradicts the transitivity. On the other hand, if $S\cong2.S_4$, choose an element of order $8$.\footnote{For each  element $s$ of order $8$ in $\SL_2(q)$, we have $\tr(s)^2-2=\tr(s^2)=0$.} Its trace is $\pm\tau$. Since $-I_2\in S$, replacing this element by its negative if necessary yields an element $s\in S$ with trace $\tau$. We then deduce $M^{-1}sM\in\mathcal E_L(\tau)$, contradicting the second computational equality. Both exceptional cases are excluded, so $S/Z$ is contained in the normalizer of a nonsplit torus.

It remains to choose $D$ and prove the asserted normality and index.
Lifting the nonsplit torus from $\PSL_2(q)$ gives a nonsplit maximal
torus $T\cong C_{q+1}$ of $\SU_2(q)$ with $S\leq N_{\SU_2(q)}(T)$.
If $S$ is cyclic, we can choose $T$ to contain $S$ and take $D=S$.
If $S$ is noncyclic and $S\not\cong Q_8$, then $D=S\cap T$ is the
unique cyclic subgroup of index two in $S$. Thus, in both cases, $D$ is
characteristic in $S$ and contains $Z$. By
\eqref{eqn_one_factor_centralizer}, $\overline{D\circ I_2}$ is normal in
$\overline G$; since $\overline H$ is also normal, so is their product
$\overline{D\circ L}$. Equation \eqref{eqn_one_factor_normalizer} gives
\[
 [\overline G:\overline{D\circ L}]
 =[\overline G:\overline{S\circ L}]\,[S:D]\mid2e.
\]
Up to conjugation in the first factor, we can take $T=\la d_\xi\ra$ without loss of generality. 
Finally, suppose that $S\cong Q_8$. By transitivity and
\eqref{eqn_one_factor_normalizer}, we have $(q+1)\mid 480e$. Since
$p\geq239$ and $e$ is odd, this forces $e=1$, and hence
$\overline G=\overline{S\circ L}$. Choose any $D\cong C_4$ in $S$.
Then $\overline{D\circ L}$ is normal of index two in $\overline G$.
The centralizer $T=C_{\SU_2(q)}(D)$ is a nonsplit maximal torus normalized
by $S$, and $S\cap T=D$. After conjugating the first factor, we may again
take $T=\la d_\xi\ra$. This completes the proof.
\end{proof}

Write $D=\la d_{\xi^k}\ra$, where $k\mid (q+1)$, and  set
$J:=D\circ L$ and $\overline J:=\overline{D\circ L}$. Since $Z\leq D\cap L$ and  $J$ acts faithfully on $\{\la\bfv_M\ra: M\in \mathcal{M},\, \det(M)=1 \}$, we have 
\[
 |\overline J|=\frac{60(q+1)}k,\qquad
 \overline J\unlhd\overline G,\qquad
 [\overline G:\overline J]\mid2e.
\]
Recall that $\nu\in\blacksquare_q$ is fixed. Choose
$\delta\in\F_{q^2}^{\times}$ such that $\delta^q=-\delta$ and
$\delta^2=\nu$.

\begin{proposition}\label{prop_2A5_prime_field}
	Assume  that \eqref{eqn_A5_full_span} holds and Case \textup{(III)} occurs.  Let $P=\la \bfv_M\ra\in \cB$. Set
\[r=|(M^{-1}DM\cap L)/Z|,\] and let $m$ be the number of
$\overline J$-orbits on $\cB$. Then $q$ is prime,
$q\equiv-1\pmod{120}$, and 
\[\mathcal{T}_{q,k}= \{ t\in \F_q: 2-t\delta (y-y^{-1})\in \square_q \text{ for every $y\in \la \xi^k\ra$}\} \setminus (\{0\} \cup \{t:1+\nu t^2=0\}) \ne \emptyset.\]
Moreover, we have 
\[ m\in\{1,2\},\, \quad r\in\{1,2,3,5\}, \qquad k=\frac{60m}{r}.\]  For any two distinct points $\la\bfv_M\ra$ and
$\la\bfv_N\ra$ in the same $\overline J$-orbit, we have
\[
B(\bfv_M,\bfv_N)\in\square_q.
\]
 Finally,
\[
 q<\left(\frac{120m}{r}\right)^2,\qquad 2kr\mid (q+1).
\]
\end{proposition}

\begin{proof}
	Since the chosen representative $\mathbf v_M$ has first coordinate $1$ by \eqref{eqn_hyperbolic_normalized_point}, we have
	\begin{equation}\label{eqn_Stabequal}
		\Stab_{\overline{J}}(\la \bfv_M\ra)=\Stab_{J}(\bfv_M)=(M^{-1}DM\cap L)/Z,
	\end{equation} 
	because $dMg^{-1}=M$ if and only if $ g=M^{-1}dM.$  This is a cyclic subgroup of $A_5$ since $D$ is cyclic, so $r\in\{1,2,3,5\}$.  Every $\overline J$-orbit of $\cB$ has length $60(q+1)/(kr)$. Since $\overline J\unlhd\overline G$ and $\overline G$ acts transitively on $\cB$, we have
\[
 k=\frac{60m}{r},\qquad m\leq2e.
\]
It follows that $k$ is even. The inclusion
$(M^{-1}DM\cap L)/Z \leq M^{-1}DM/Z$ implies that $2kr\mid (q+1)$.

The equality  \eqref{eqn_Stabequal} allows us to apply Theorem
\ref{thm_square_class}, so the homomorphism $\epsilon$ is trivial on $H$. For $y\in\la\xi^k\ra\setminus\{1\}$,  by Lemma \ref{lem_xisquare} and
$q\equiv3\pmod4$, we obtain 
$B(\bfv_M,(d_y\circ I_2)\bfv_M)=2-y-y^{-1}\in\square_q$. Hence  the homomorphism $\epsilon$ is also trivial on $D\circ L$, and then for any two distinct points $\la\bfv_M\ra,\la\bfv_N\ra$ in the same $\overline J$-orbit, $B(\bfv_M,\bfv_N)\in\square_q.$

We first observe that, for every $M$, there exists an element $g\in L$ of order four such that $MgM^{-1}$ is neither diagonal nor antidiagonal.\footnote{The elements of order four generate $L$.
If all their conjugates by $M$ were diagonal or antidiagonal,
then $MLM^{-1}$ would preserve the pair of coordinate lines.
Since $L$ is perfect, its induced action on this pair would be trivial.
Thus both lines would be invariant, contrary to absolute irreducibility.}  Choose such an element $g\in L$, and write
\[
MgM^{-1}=
\begin{pmatrix}
	\delta t&b\\
	-b^q&-\delta t
\end{pmatrix},
\qquad tb\ne0.
\]
Since it lies in $\SU_2(q)$, we have 
\[
bb^q=1+\nu t^2\ne0.
\]
Moreover, $d_y\circ g^{-1}$ cannot fix $P$ for any
$y\in\la\xi^k\ra$, since otherwise
$MgM^{-1}=d_y^{-1}$, contradicting $b\ne0$. Hence
\[
2-t\delta(y-y^{-1})\in\square_q
\qquad\text{for every }y\in\la\xi^k\ra,
\]
and therefore $t\in\mathcal T_{q,k}$.

 Since $4+4\nu t^2=4bb^q\ne0$ and every character value is $1$,
 Lemma \ref{lem_caseI_torus_sum}, applied with $\alpha=-\delta t$ and
 $\beta=2$, gives $(q+1)/k<2\sqrt q$, and hence $q<4k^2$. By \eqref{eqn_hyperbolic_prime_congruence}, we have $p\geq239$. If $e>1$, then $p^e>(240e)^2$, contradicting $k\leq120e/r$ and $q<4k^2$.
Thus $e=1$, so $m\leq2$, and substituting $k=60m/r$ gives the stated
bound. This completes the proof.
\end{proof}

\subsection{Computer-assisted elimination of the remaining cases}\label{sec:magma}

The remaining finite list of parameters is eliminated by exact
finite-field arithmetic in Magma~V2.28-9~\cite{Magma}.

\begin{proposition}\label{prop_A5_scalar_bound}
In the external-line case with $(q+1)/k>2$, let $\mathcal A_{q,k}$ be defined by \eqref{eqn_caseI_admissible_scalars}. The condition 
$\mathcal A_{q,k}\ne\varnothing$ leaves exactly
the following possibilities:
\[
\begin{array}{c|l}
k&q\\ \hline
60&179,239,359,419,479,599,659,719,839,1019,1259\\
120&359,479,599,719,839,1319,1439,1559,2039,2399.
\end{array}
\]
When $(q+1)/k\leq2$, the only remaining pair is $(q,k)=(239,120)$.

In the one-factor case \textup{(III)} of Subsection \ref{sect:hyperbolic}, the conditions 
in Proposition \ref{prop_2A5_prime_field} leave exactly the following possibilities:
\[
\begin{array}{c|c|c|l}
m&r&k&q\\ \hline
1&1&60&239,359,479,599,719,839\\
1&2&30&239,479\\
1&3&20&239\\
2&1&120&239,479,719,1439,2399\\
2&2&60&239,479,719\\
2&3&40&239,479,719\\
2&5&24&239.
\end{array}
\]
\end{proposition}

\begin{proof}
By Corollary \ref{cor_caseI_prime_branches} and Proposition
\ref{prop_2A5_prime_field}, we obtain $q<(120m-1)^2$ in the external-line case and $q<(120m/r)^2$ in the hyperbolic case \textup{(III)}, where $m\in\{1,2\}$ and
$r\in\{1,2,3,5\}$. An exhaustive check of the corresponding square conditions in Magma \cite{Magma} gives the displayed lists of parameters. 
\end{proof}

\begin{proposition}\label{prop:A5-finite-verification}
Assume the hypotheses of Section \ref{sec:4} hold. Then the external-line case and Case \textup{(III)} of the hyperbolic case do not occur.
\end{proposition}

\begin{proof}[Computer-assisted verification]
 In each case, the two embedding classes are conjugate in $\PGO_5(q)$.
We therefore fix one representation $L$ in each model, with
$L\cong A_5$ in the external-line case and $L\cong2.A_5$
in the hyperbolic case. For $q=239$, an exhaustive Magma computation
tests the complete candidates, including compatibility between
distinct $\overline J$-orbits when $m=2$, and finds no BLT-set of size $240$.
For the remaining values of $q$, we use the following procedures,
implemented in Magma~\cite{Magma}.
\begin{description}
\setlength{\itemsep}{2pt}
\setlength{\parsep}{0pt}

\item[(i)] \emph{The external-line case with $q\ne239$.}
Consider the $20$ pairs $(q,k)$ in Proposition
\ref{prop_A5_scalar_bound} for which $(q+1)/k>2$. 
 Choose a representative $(u,\mathbf w)$ with $u^{q+1}=-2b_5\nu$ and $Q_{A_5}(\mathbf w)=2b_5\nu$. Fix a Klein four-group $\{1,\tau_1,\tau_2,\tau_3\}\leq H$. For each $i$, choose a nonzero vector $\mathbf w_i$ spanning the $(+1)$-eigenspace of $\tau_i$. Then these three eigenspaces form an orthogonal decomposition of $W$. Consequently, every admissible $\bfw$ determines
\[
a_i=
-\frac{Q_{A_5}(\bfw+\tau_i\bfw)}{2b_5}
\qquad (1\leq i\leq3),
\]
where
\[
(a_1,a_2,a_3)
 \in(\mathcal A_{q,k}\cap\square_q)^3,
\qquad
a_1+a_2+a_3=-4\nu.
\]
Let $\bfw=x_1\bfw_1+x_2\bfw_2+x_3\bfw_3$. Then reconstruct the vector $\bfw$ from the conditions above. By \eqref{eqn_caseI_square_condition},  retain $\bfw$ only if
\[
 -\frac{Q_{A_5}(\bfw+h\bfw)}{2b_5}
 \in\mathcal A_{q,k}
\]
for every element $h\in H$ of order five.\\ 
\textbf{Result}: No vector survives.

\item[(ii)] \emph{The hyperbolic case \textup{(III)} with $q\ne239$.}
Choose
\[
 Q=\langle g_1,g_2\rangle\cong Q_8\leq L\cong 2.A_5,
 \qquad g_3=g_1g_2.
\]  
For a matrix $M\in \mathcal{M}$ with $\det(M)=1$, define $t_i\in\F_q$ by
\[
 (Mg_iM^{-1})_{11}=\delta t_i.
\]
There is a bijection between the set of the left $\la d_\xi\ra$-cosets in $\SU_2(q)$ and $\{(t_1,t_2,t_3)\in \F_q^3:  -\nu(t_1^2+t_2^2+t_3^2)=1 \}$.\footnote{On $\mathcal{M}$, we consider the $3$-dimensional subspace $\mathcal{M}_0$ consisting of the matrices of trace $0$. Equip it with the inner product $\la X,Y\ra=\frac{1}{2}\tr(XY)$. The quaternion relations in $Q_8$ imply that $g_1,g_2,g_3$ form an orthogonal basis of $\mathcal{M}_0$. Set $\Delta=\operatorname{diag}(\delta,-\delta)$. Then we have $\tr(M^{-1} \Delta M g_i )=2\nu t_i$ for each $i$. It follows that $\la d_\xi\ra  \backslash\SU_2(q)\leftrightarrow \{(t_1,t_2,t_3)\in \F_q^3: -\nu(t_1^2+t_2^2+t_3^2)=1\}$.} Under this bijection, the $\operatorname{Ad}(L)$-orbits 
correspond to the double cosets in $\la d_\xi\ra  \backslash\SU_2(q)/L$.

For each remaining quadruple $(q,m,r,k)$ in Proposition
\ref{prop_A5_scalar_bound}, enumerate the
$\operatorname{Ad}(L)$-orbits of the solutions of this equation and reconstruct representative matrices $M\in\SU_2(q)$
for the double cosets in $T\backslash\SU_2(q)/L$.

Put $Y=\langle\xi^k\rangle$, and choose a set
$\mathcal R\subseteq L$ containing one representative from each pair
$\{h,-h\}$.  Retain $M$ only if exactly $r$ pairs
$(y,h)\in Y\times\mathcal R$ satisfy
\[
 d_yMh=M,\quad y\in \la \xi^k\ra,\, h\in \mathcal{R}, 
\]
and, for every other pair,
\[
 2-\tr(M^{-1}d_yMh)\in\square_q.
\]
No matrix is retained.
\end{description}

Thus no candidate remains in either model. Hence neither the
external-line case nor Case \textup{(III)} occurs.
\end{proof}

\begin{proposition}\label{prop_A5_case}
Let $\cB$ satisfy the hypotheses of Theorem \ref{thm_main}, and assume
that $q\geq71$ and the full-span hypothesis holds. Then $\Gso$ has no
exceptional $A_5$-composition factor.
\end{proposition}

\begin{proof}
Lemma \ref{lem_A5_geometric_reduction} reduces the analysis of an
exceptional $A_5$-composition factor to the external-line case and Cases
\textup{(I)}--\textup{(III)}. Lemmas
\ref{lem_A5_twisted_diagonal_exclusion} and
\ref{lem_A5_product_exclusion} exclude Cases \textup{(I)} and
\textup{(II)}, while Proposition \ref{prop:A5-finite-verification}
excludes the external-line case and Case \textup{(III)}.
\end{proof}

\begin{proof}[Proof of Theorem \ref{thm_main}]
For $q<71$, the classifications in
\cites{LawPenttila2003,Betten2013} and the automorphism-group data in
\cite{SmallBLT} give the result. Assume that $q\geq71$.

Since $\overline G$ is non-solvable and
$\overline G/\overline{\Gso}$ is solvable, $\Gso$ has a non-abelian composition factor. Since $G$ is reducible, it preserves one of the six geometries listed in Theorem \ref{thm_SO_5(q)}. By that theorem and Lemma
\ref{lem_subPSL}, every such factor is either $\PSL_2(q_0)$, where $q_0$
is a power of $p$, or $A_5$ with $p\ne5$.

If a factor $\PSL_2(q_0)$ occurs, we treat the six geometries in Propositions \ref{prop_case1}--\ref{prop_case6}, respectively. The first
three cases are impossible, the next two cases yield a classical BLT-set, and the last yields a classical or Kantor semifield BLT-set.

If the non-abelian factor is $A_5$ and $\cB$ lies in a hyperplane, Theorem
\ref{thm_charact} applies. Otherwise \eqref{eqn_A5_full_span} holds, and
Proposition \ref{prop_A5_case} excludes the case.
\end{proof}

\section{Concluding remarks}\label{sec:5}
Theorem \ref{thm_main} may also be rephrased by using the language of flock
generalized quadrangles. A flock of a quadratic cone determines an elation generalized
quadrangle of order $(q^2,q)$, and Knarr's construction produces the same
geometry from the associated BLT-set \cites{Thas1987,Knarr1992}. Under the
duality between $\cQ(4,q)$ and $W(3,q)$, the $q+1$ points of a BLT-set
correspond to $q+1$ totally isotropic lines of $W(3,q)$. Their lifts in $W(5,q)$ are the
totally isotropic planes through the distinguished point in the Knarr model.
Thus transitivity on the BLT-set corresponds to transitivity on this family
of totally isotropic planes in the Knarr model.

 Together with the result of Nelson and Penttila
\cite{NelsonPenttila2013}, we complete the classification of
the transitive BLT-sets with a non-solvable stabilizer in Theorem \ref{cor:full-classification}. However, the classification of BLT-sets admitting a solvable transitive full automorphism group remains open.  Betten \cite{Betten2009} constructed a transitive BLT-set over $\F_{41}$ whose full automorphism group is
\[
C_{21}\rtimes(C_2\times C_2)\cong S_3\times D_{14},
\]
which is a solvable group of order $84$. Betten's construction also recovers several known examples, but no infinite family containing this example over $\F_{41}$ is known. It remains open whether this example actually belongs to an infinite family. More generally, the classification of transitive BLT-sets admitting a solvable full automorphism group remains open.

\par\medskip
\noindent\textbf{Declaration on the use of generative AI tools.} 
During the preparation of this paper, the authors used GPT-5.6 Sol Pro and GPT-6 Astra Pro to assist in developing some of the arguments in Section~4, including those for the cases \(q=79\) and \(q=239\) in Lemma~4.1. 
The Magma \cite{Magma} programs used for the computations in Section~\ref{sec:magma} were developed primarily with the assistance of these tools. These tools were also used to improve the language and exposition of the manuscript.
The authors independently reviewed all AI-assisted material and code and verified the mathematical arguments, computational input ranges, and results. The authors take full responsibility for the content of the paper.
\par\medskip

\noindent\textbf{Acknowledgements.} 
Qing Xiang was supported by the National Key Research and Development Program of China (Grant No. 2025YFA1017700) and the National Natural Science Foundation of China (Grant No. 12131011). Weicong Li was supported by the National Natural Science Foundation of China (Grant No. 12301422).
The authors thank Professor Tao Feng for helpful comments and suggestions. The first author would like to thank Dr. Yanzhou Zhu for pointing out the use of first cohomology of $\SL_2(q)$ and $\Omega_3(q)$ to streamline the proofs of Propositions \ref{prop_case1} and \ref{prop_case3}.

\bibliographystyle{plain}


\begin{thebibliography}{10}
	
	\bibitem{BLT1990}
	L.~Bader, G.~Lunardon, and J.~A. Thas.
	\newblock Derivation of flocks of quadratic cones.
	\newblock {\em Forum Math.}, \textbf{2}(2):163--174, 1990.
	
	\bibitem{BOP2004some}
	L.~Bader, C.~M. O'Keefe, and T.~Penttila.
	\newblock Some remarks on flocks.
	\newblock {\em J. Aust. Math. Soc.}, \textbf{76}(3):329--343, 2004.
	
	\bibitem{BDEW1999}
	R.~D. Baker, J.~M. Dover, G.~L. Ebert, and K.~L. Wantz.
	\newblock Hyperbolic fibrations of {$\textup{PG}(3,q)$}.
	\newblock {\em European J. Combin.}, \textbf{20}(1):1--16, 1999.
	
	\bibitem{BEP2005}
	R.~D. Baker, G.~L. Ebert, and T.~Penttila.
	\newblock Hyperbolic fibrations and {$q$}-clans.
	\newblock {\em Des. Codes Cryptogr.}, \textbf{34}(2--3):295--305, 2005.
	
	\bibitem{BGR2010}
	J.~Bamberg, M.~Giudici, and G.~F. Royle.
	\newblock Every flock generalized quadrangle has a hemisystem.
	\newblock {\em Bull. Lond. Math. Soc.}, \textbf{42}(5):795--810, 2010.

	\bibitem{Betten2009}
	A.~Betten.
	\newblock A class of transitive {BLT}-sets.
	\newblock {\em Note Mat.}, \textbf{29}(suppl. 1):1--9, 2009.

	\bibitem{Betten2013}
	A.~Betten.
	\newblock Rainbow cliques and the classification of small {BLT}-sets.
	\newblock In {\em ISSAC 2013}, pages 53--60. ACM, 2013.
	
	\bibitem{SmallBLT}
	A.~Betten.
	\newblock A classification of {BLT}-sets of small order.
	\newblock \url{https://www.math.colostate.edu/~betten/blt.html}, 2009.

    \bibitem{Magma}
	W.~Bosma, J.~Cannon, and C.~Playoust.
	\newblock The {M}agma algebra system. {I}. The user language.
	\newblock {\em J. Symbolic Comput.}, \textbf{24}(3--4):235--265, 1997.
    
	\bibitem{BHR407}
	J.~N. Bray, D.~F. Holt, and C.~M. Roney-Dougal.
	\newblock {\em The maximal subgroups of the low-dimensional finite classical
		groups}, volume~\textbf{407} of {\em London Mathematical Society Lecture Note Series}.
	\newblock Cambridge University Press, Cambridge, 2013.
	\newblock With a foreword by Martin Liebeck.

    
    	\bibitem{ClineParshallScott1975}
	E.~Cline, B.~Parshall, and L.~Scott.
	\newblock Cohomology of finite groups of {Lie} type, {I}.
	\newblock {\em Publ. Math. Inst. Hautes \'{E}tudes Sci.},
	\textbf{45}:169--191, 1975.
    
    	\bibitem{ATLAS}
	J.~H. Conway, R.~T. Curtis, S.~P. Norton, R.~A. Parker, and R.~A. Wilson.
	\newblock {\em Atlas of finite groups}.
	\newblock Clarendon Press, Oxford, 1985.
    
	\bibitem{Dickson1904}
	L.~E. Dickson.
	\newblock {\em Linear groups: {W}ith an exposition of the {G}alois field
		theory}.
	\newblock Dover Publications, Inc., New York, 1958.
	\newblock With an introduction by W. Magnus.

  \bibitem{Hirschfeld1998}
    J.~W.~P. Hirschfeld.
    \newblock {\em Projective Geometries over Finite Fields}.
    \newblock 2nd ed., Oxford Math. Monogr., Clarendon Press,
    Oxford Univ. Press, Oxford, 1998.

	
	\bibitem{Huppert}
	B.~Huppert.
	\newblock {\em Endliche Gruppen. I}.
	\newblock Die Grundlehren der Mathematischen Wissenschaften, volume~134.
	Springer-Verlag, Berlin, 1967.
	
   \bibitem{BrauerAtlas}
	C.~Jansen, K.~Lux, R.~Parker, and R.~A. Wilson.
	\newblock {\em An atlas of Brauer characters}.
	\newblock London Mathematical Society Monographs. New Series, volume~11.
	Clarendon Press, Oxford, 1995.
	
	\bibitem{Kantor1991GenBLT}
	W.~M. Kantor.
	\newblock Generalized quadrangles, flocks, and {BLT} sets.
	\newblock {\em J. Combin. Theory Ser. A}, \textbf{58}(1):153--157, 1991.
    
      
	\bibitem{Kantor1982Rank3}
	W.~M. Kantor and R.~A. Liebler.
	\newblock The rank {$3$} permutation representations of the finite classical
	groups.
	\newblock {\em Trans. Amer. Math. Soc.}, \textbf{271}(1):1--71, 1982.
	
	\bibitem{King2005Survey}
	O.~H. King.
	\newblock The subgroup structure of finite classical groups in terms of
	geometric configurations.
	\newblock In {\em Surveys in combinatorics 2005}, volume \textbf{327} of {\em London
		Math. Soc. Lecture Note Ser.}, pages 29--56. Cambridge Univ. Press,
	Cambridge, 2005.
	
	\bibitem{KL129}
	P.~Kleidman and M.~Liebeck.
	\newblock {\em The subgroup structure of the finite classical groups}, volume
	\textbf{129} of {\em London Mathematical Society Lecture Note Series}.
	\newblock Cambridge University Press, Cambridge, 1990.
	
	\bibitem{Knarr1992}
	N.~Knarr.
	\newblock A geometric construction of generalized quadrangles from polar spaces
	of rank three.
	\newblock {\em Results Math.}, \textbf{21}(3--4):332--344, 1992.
	
	\bibitem{Law2004Con}
	M.~Law and T.~Penttila.
	\newblock Construction of {BLT}-sets over small fields.
	\newblock {\em European J. Combin.}, \textbf{25}(1):1--22, 2004.

	\bibitem{LawPenttila2003}
	M.~Law and T.~Penttila.
	\newblock Classification of flocks of the quadratic cone over fields of
	order at most $29$. \newblock {\em Adv. Geom.}, 	\textbf{3}(Special Issue):S232--S244, 2003.
	
	\bibitem{Finitefield}
	R.~Lidl and H.~Niederreiter.
	\newblock {\em Finite fields}, volume~\textbf{20} of {\em Encyclopedia of Mathematics and its Applications}.
	\newblock Cambridge University Press, Cambridge, 2nd edition, 1997.
	\newblock With a foreword by P. M. Cohn.
	
	\bibitem{Mitchell1914}
	H.~H. Mitchell.
	\newblock The subgroups of the quaternary abelian linear group.
	\newblock {\em Trans. Amer. Math. Soc.}, \textbf{15}(4):379--396, 1914.
	
	\bibitem{Nelson-thesis}
	E.~M. Nelson.
	\newblock {\em The group of the {M}ondello {BLT}-sets}.
	\newblock ProQuest LLC, Ann Arbor, MI, 2012.
	\newblock Thesis (Ph.D.)--Colorado State University.
	
	\bibitem{NelsonPenttila2013}
	E.~M. Nelson and T.~Penttila.
	\newblock On transitive irreducible {BLT}-sets.
	\newblock {\em Discrete Math.}, \textbf{313}(12):1284--1288, 2013.

	\bibitem{Payne1996}
	S.~E. Payne.
	\newblock The fundamental theorem of {$q$}-clan geometry.
	\newblock {\em Des. Codes Cryptogr.}, \textbf{8}(1--2):181--202, 1996.

    \bibitem{Serre}
     J.~P. Serre.
     \newblock \emph{Finite Groups: An Introduction,} 
     \newblock Surveys of Modern Mathematics, Vol.~10, International Press, 2016.  
	
	\bibitem{Thas1987}
	J.~A. Thas.
	\newblock Generalized quadrangles and flocks of cones.
	\newblock {\em European J. Combin.}, \textbf{8}(4):441--452, 1987.
	
	\bibitem{Thas1993CharFTW}
	J.~A. Thas.
	\newblock A characterization of the {F}isher-{T}has-{W}alker flocks.
	\newblock {\em Simon Stevin}, \textbf{67}(3--4):219--226, 1993.
	
	\bibitem{van2010}
	E.~R. van Dam and M.~Muzychuk.
	\newblock Some implications on amorphic association schemes.
	\newblock {\em J. Combin. Theory Ser. A}, \textbf{117}(2):111--127, 2010.
	
	\bibitem{Walker1976}
	M.~Walker.
	\newblock A class of translation planes.
	\newblock {\em Geometriae Dedicata}, \textbf{5}(2):135--146, 1976.

	\bibitem{Zsigmondy}
	K.~Zsigmondy.
	\newblock Zur {T}heorie der {P}otenzreste.
	\newblock {\em Monatsh. Math. Phys.}, \textbf{3}:265--284, 1892.
	
\end{thebibliography}

\end{document}